\documentclass[12pt,reqno]{amsart}

\usepackage{amsmath,amsthm,amsfonts,amssymb,amsopn}
\usepackage{fancyheadings} 
\usepackage{setspace}                       
\usepackage{graphicx}   
\usepackage{caption}                               
\usepackage{subcaption} 
\usepackage{tikz}
\usepackage{enumerate}
\usepackage{xfrac}
\usepackage{breqn}
\usepackage{amssymb}
\usepackage{float} %minipage for 
\usepackage{pgfplots}
\usepackage{graphicx}
\usetikzlibrary{decorations.pathmorphing}

\theoremstyle{plain}

\theoremstyle{plain}
\newtheorem{thm}{Theorem}

\newtheorem{theorem}{Theorem}

\newtheorem{proposition}[thm]{Proposition}
\newtheorem{lemma}{Lemma}

\newtheorem{remark}{Remark}

\newcommand{\R}{\Bbb{R}}
\newcommand {\Z} {\mathbb Z}

\numberwithin{equation}{section}

\newcommand \norm[1] {\left \lVert #1\right \rVert}
\newcommand {\lb} {\left(}
\newcommand {\rb} {\right)}
\newcommand \lsb {\left[}
\newcommand \rsb {\right]}

\newcommand{\eps}{\varepsilon}

\newcommand {\p} [1]{\partial_{#1}}

\newcommand {\solsp}[3] {C^{#1}\lb \lsb -#2,#2\rsb,H^{#3}\rb}
\newcommand \elt[1] {\norm{#1}_{\ell^2}}
\newcommand \Li[1]{\norm{#1}_{L^\infty}}

\newcommand {\uh} [2] {\hat u \left( #1, #2 \right)}

\newcommand \jb[1] {\left\langle #1 \right\rangle }
\newcommand \rint {\int_{-\infty}^\infty}
\newcommand \sint [1]{\int_{-#1}^{#1}}

\newcommand \ux[2] {u^{(1)}_{#1,#2}}
\newcommand \uy[2] {u^{(2)}_{#1,#2}}
\newcommand \vx[2] {v^{(1)}_{#1,#2}}
\newcommand \vy[2] {v^{(2)}_{#1,#2}}
\newcommand \Ux[2] {U^{(1)}_{#1,#2}}
\newcommand \Uy[2] {U^{(2)}_{#1,#2}}
\newcommand \Vx[2] {V^{(1)}_{#1,#2}}
\newcommand \Vy[2] {V^{(2)}_{#1,#2}}

\newcommand \Up[1] {
	\IfEqCase{#1}{
		{0}{A}
		{1}{A_0}
		{2}{A_1}
	}
}

\newcommand \At {A\lb \xi, \eta, \tau\rb}

\newcommand \xe[2] {\uppercase{#1}\lb \xi#2\eps,\eta\rb}
\newcommand \ye[2] {\uppercase{#1}\lb \xi,\eta#2\eps^2\rb}
\newcommand \We [1] {W^{(#1)}}
\newcommand \Ue [1] {U^{(#1)}}
\newcommand \Le [1] {L^{(#1)}}

\newcommand \UV[2] {#1^{(#2)}}

\newcommand {\ds}{\displaystyle\sum}

\newcommand \av[3]{a^{#1}_{#2,#3}}
\newcommand \efun [1]{#1_\eps}
\newcommand {\pxi} [1] {\p \xi^{-#1}}

\begin{document}

\title[Justification of KP-II for two-dimensional FPU systems]{Justification of the KP-II approximation\\
	 in dynamics of two-dimensional FPU systems}

\author{Nikolay Hristov}
\address{Department of Mathematics and Statistics, McMaster University, Hamilton, Ontario, Canada, L8S 4K1}

\author{Dmitry E. Pelinovsky}
\address{Department of Mathematics and Statistics, McMaster University, Hamilton, Ontario, Canada, L8S 4K1}
\email{dmpeli@math.mcmaster.ca}

\begin{abstract}
	Dynamics of the Fermi--Pasta--Ulam (FPU) system on a two-dimensional square lattice is considered in the limit of small-amplitude long-scale waves with slow transverse modulations. In the absence of transverse modulations, dynamics of such waves, even at an oblique angle with respect to the square lattice, is known to be described by the Korteweg--de Vries (KdV) equation. For the three basic directions (horizontal, vertical, and diagonal), we prove that the modulated waves are well described by the Kadomtsev--Petviashvili (KP-II) equation. The result was expected long ago but proving rigorous bounds on the approximation error turns out to be complicated due to the nonlocal terms of the KP-II equation and the vector structure of the FPU systems on two-dimensional lattices. We have obtained these error bounds by extending the local well-posedness result for the KP-II equation in Sobolev spaces and by controlling the error terms with energy estimates. The bounds are useful in the analysis of transverse stability of solitary and periodic waves in two-dimensional FPU systems due to many results available for the KP-II equation. 
\end{abstract}

\maketitle
\date{}

\section{Introduction}

A Fermi-Pasta-Ulam (FPU) system is formed by particles connected
to their nearest neighbours by nonlinear springs. If the particles are 
organized in a one-dimensional chain, we can label the position variable 
of the $j$-th particle by $q_j$. If we 
denote the corresponding velocity variable by $p_j$, then the FPU system is generated by the Hamiltonian function in the form 
\begin{equation}
\label{Hamiltonian-1D}
H(p,q) = \sum_{j \in \mathbb{Z}} \frac{1}{2} p_j^2 + V(q_{j+1}-q_j),
\end{equation}
where $V(q)$ is the potential energy. A useful approximation 
to dynamics of small-amplitude long-scale waves in the FPU chain
with smooth $V$ satisfying $V'(0) = 0$, $V''(0) > 0$, and $V'''(0) \neq 0$ 
is given by the Korteweg--de Vries (KdV) equation, which is a remarkable 
model due to integrability, stability of periodic and solitary waves, and global existence of solutions in Sobolev spaces (see \cite{KdV-review} for review).

Bounds on the approximation error between solutions of the FPU system and the KdV equation were obtained by G. Schneider and C.E. Wayne 
in a conference proceeding \cite{SW1} as an exercise related to 
the justification technique the
authors had developed for the water wave problem \cite{SW2}. 
The same approximation appeared to be very useful in the context of 
stability of FPU solitary waves and was studied comprehensively
in the series of papers by G. Friesecke and R. L. Pego \cite{FP1,FP2,FP3,FP4}. 
It was also justified in the context of normal forms and KAM theory 
for metastability and recurrency of the FPU systems \cite{BP,Ponno}.

More recently, the same approximation but with other 
versions of the KdV equation was justified for FPU systems with Hertzian potentials \cite{DP} and with pure anharmonic powers \cite{Khan}. In \cite{Khan}, the KdV approximation was extended to logarithmically long time scales provided the global dynamics of the generalized KdV equation is well defined in Sobolev spaces of higher regularity. 
The KdV approximation was also used in the context of periodic 
waves and for FPU systems with nonlocal interactions \cite{FM,Herr1}. 
Higher-order corrections to the KdV approximation were studied in \cite{GPR}, 
where it was shown that the second-order corrections are spanned by 
two members of the integrable KdV hierarchy  whereas the third-order corrections 
can only be spanned by three members of the KdV hierarchy under a constraint on parameters of the FPU system. Review of results on nonlinear waves in one-dimensional FPU chains can be found in \cite{Vainshtein}.

{\em The purpose of this work is to consider dynamics of the two-dimensional FPU systems and to justify the two-dimensional KdV approximation given by the 
Kadomtsev--Petviashvili (KP-II) equation.} Similar to the KdV equation, 
the KP-II equation is remarkable due to integrability \cite{KP-book}, stability of periodic and solitary waves \cite{Har,Miz}, and global existence of solutions in Sobolev spaces \cite{Hadac,Molinet}.  

The first formal derivation of the KP-II equation was performed in \cite{Z} for the scalar extension of the two-dimensional FPU systems with the total energy of the form
\begin{equation}
\label{Hamiltonian-2D}
H = \sum_{(j,k) \in \mathbb{Z}^2} \frac{1}{2} p_{j,k}^2 + 
\frac{1}{2} (q_{j+1,k}-q_{j,k})^2 + \frac{1}{2} \eps^2 (q_{j,k+1} - q_{j,k})^2 + \frac{1}{3} \alpha (q_{j+1,k} - q_{j,k})^3,
\end{equation}
where $\eps^2$ is a small parameter of slow transverse modulations 
in the $k$-direction of the dominant wave propagating in the $j$ direction and $\alpha$ is the parameter for the cubic interaction potential. A similar scalar model was considered in the recent work \cite{PP21}, where the KP-II equation in the periodic domain was rigorously justified (among other integrable models) as the normal form for metastability phenomenon in two-dimensional rectangular lattices. 
Existence of two-dimensional solitary waves \cite{W1} and breathers (space-localized and time-periodic solutions) \cite{W2,W3} was also studied 
in the scalar two-dimensional FPU lattices. Extensions of the scalar two-dimensional FPU models that include harmonic interactions between the first and second nearest neighbors were considered in \cite{Astakhova,Ioan}. 
Applications of the scalar FPU models to the two-dimensional square-packed granular arrays were studied numerically and experimentally in \cite{LFD}, where propagation of a modulated one-dimensional solitary wave was observed. 
Periodic triangular lattices were compared with square lattices 
in the numerical study in \cite{B,BG}, where it was shown that non-square lattices do not have integrable approximations leading to the energy localization phenomenon.

The scalar problem can not be realized mechanically in the dynamics of 
particles organized in a two-dimensional square lattice and connected by the nonlinear springs. On the other hand, in the context of mechanical models of the elasticity theory, the long-wave reductions were used to derive the one-dimensional KdV equation (rather than the two-dimensional KP-II equation) 
from the vector two-dimensional FPU systems. 

A two-dimensional square FPU lattice was considered in \cite{FM-2003} with additional springs along the diagonals and a pair of potential functions, one for the horizontal and vertical displacements and the other one for the diagonal displacements. Existence of small-amplitude supersonic longitudinal solitary waves propagating along the horizontal direction was proven in \cite{FM-2003}. Surprisingly, the result holds even when the potential functions are quadratic, 
due to the geometric nonlinearity of the lattice. Nonlinear couplings were included in the consideration of the same model in \cite{CH}, where the solitary wave propagated under an arbitrary angle with respect to the square lattice. 

Another work can be found in \cite{ST}, where propagation 
of rings in two-dimensional lattices was analyzed in the linear approximation 
and compared rigorously with the approximation of the linearized KdV equation (rather than with the linearized KP-II equation). 

{\em Compared to the previous works, we prove validity of the KP-II approximation for dynamics of transversely modulated small-amplitude long-scale waves in a vector FPU system on a two-dimensional square lattice.}

Next we describe the mechanical model considered here. Each 
particle at the $(j,k)$ site of the two-dimensional square lattice is characterized by the vectors of relative displacements $(x_{j,k}, y_{j,k})$ and relative velocities $(\dot{x}_{j,k},\dot{y}_{j,k})$. The nonlinear springs 
connecting the particles are shown on Figure \ref{fig:2Dlat}.

\begin{figure}[H]
	\centering
	\begin{tikzpicture}[
	wall/.style = {gray,fill=gray},
	mass/.style = {draw,circle,ball color=red},
	spring/.style = {decorate,decoration={zigzag, pre length=.3cm,post length=.3cm,segment length=#1}},
	]
	\coordinate (l1) at (0,2);
	\node[mass,label={[xshift=0.9cm, yshift=0.05cm]$_{j-1,k-1}$}] (m11) at (2,2) {};
	\node[mass,label={[xshift=0.9cm, yshift=0.05cm]$_{j,k-1}$}] (m21) at (4,2) {};
	\node[mass,label={[xshift=0.9cm, yshift=0.05cm]$_{j+1,k-1}$}] (m31) at (6,2) {};
	\coordinate (r1) at (8,2);
	
	\coordinate (l2) at (0,4);
	\node[mass,label={[xshift=0.9cm, yshift=0.05cm]$_{j-1,k}$}] (m12) at (2,4) {};
	\node[mass,label={[xshift=0.9cm, yshift=0.05cm]$_{j,k}$}] (m22) at (4,4) {};
	\node[mass,label={[xshift=0.9cm, yshift=0.05cm]$_{j+1,k}$}] (m32) at (6,4) {};
	\coordinate (r2) at (8,4);
	
	\coordinate (l3) at (0,6);
	\node[mass,label={[xshift=0.9cm, yshift=0.05cm]$_{j-1,k+1}$}] (m13) at (2,6) {};
	\node[mass,label={[xshift=0.9cm, yshift=0.05cm]$_{j,k+1}$}] (m23) at (4,6) {};
	\node[mass,label={[xshift=0.9cm, yshift=0.05cm]$_{j+1,k+1}$}] (m33) at (6,6) {};
	\coordinate (r3) at (8,6);

	\draw[spring=4pt] (l1) -- node[above] {} (m11);
	\draw[spring=4pt] (m11) -- node[above] {} (m21);
	\draw[spring=4pt] (m21) -- node[above] {} (m31);
	\draw[spring=4pt] (m31) -- node[above] {} (r1);
	
	\draw[spring=4pt] (l2) -- node[above] {} (m12);
	\draw[spring=4pt] (m12) -- node[above] {} (m22);
	\draw[spring=4pt] (m22) -- node[above] {} (m32);
	\draw[spring=4pt] (m32) -- node[above] {} (r2);
	
	\draw[spring=4pt] (l3) -- node[above] {} (m13);
	\draw[spring=4pt] (m13) -- node[above] {} (m23);
	\draw[spring=4pt] (m23) -- node[above] {} (m33);
	\draw[spring=4pt] (m33) -- node[above] {} (r3);
	
	\draw[spring=4pt] (2,8) -- node[above] {} (m13);
	\draw[spring=4pt] (m13) -- node[above] {} (m12);
	\draw[spring=4pt] (m12) -- node[above] {} (m11);
	\draw[spring=4pt] (m11) -- node[above] {} (2,0);
	
	\draw[spring=4pt] (4,8) -- node[above] {} (m23);
	\draw[spring=4pt] (m23) -- node[above] {} (m22);
	\draw[spring=4pt] (m22) -- node[above] {} (m21);
	\draw[spring=4pt] (m21) -- node[above] {} (4,0);
	
	\draw[spring=4pt] (6,8) -- node[above] {} (m33);
	\draw[spring=4pt] (m33) -- node[above] {} (m32);
	\draw[spring=4pt] (m32) -- node[above] {} (m31);
	\draw[spring=4pt] (m31) -- node[above] {} (6,0);
	
	\end{tikzpicture}
	\caption{A mass--spring mechanical system arranged in a square lattice}\label{fig:2Dlat}
\end{figure}
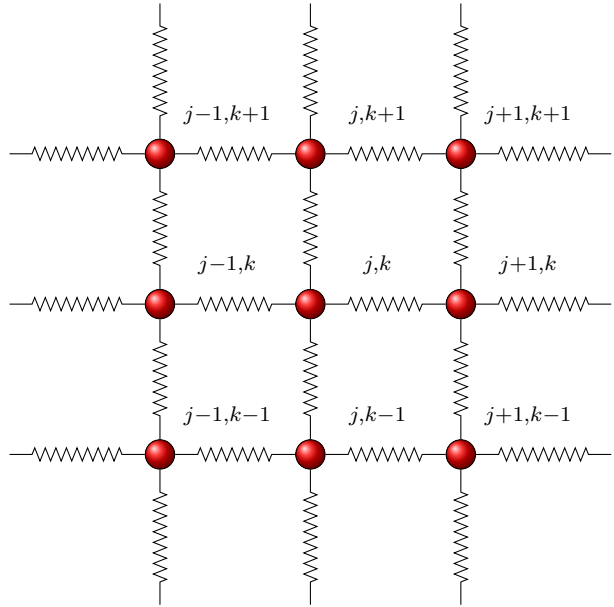

The potential energy of a single spring between two particles in the horizontal direction is defined by 
$$
V(r,s) = \frac{1}{2} (c_1^2 r^2 + c_2^2 s^2) + \frac{1}{3} \alpha_1 r^3 + \frac{1}{2} \alpha_2 r s^2,
$$
where $(r,s)$ are the relative displacements of the two particles in the $(x,y)$ coordinates, $(c_1^2,c_2^2)$ are some coefficients of the quadratic interaction potential, and $(\alpha_1,\alpha_2)$ are some  coefficients of the cubic interaction potential. The cubic terms are chosen in such a way that the  potential energy of the horizontal spring is symmetric with respect to the sign of vertical relative displacements of the particles. Due to the symmetry between horizontal and vertical springs, the total energy of the two-dimensional FPU lattice takes the form 
\begin{align}
\nonumber
H &= \frac{1}{2} \sum_{(j,k) \in \mathbb{Z}^2} \dot{x}_{j,k}^2 + \dot{y}_{j,k}^2 \\
\label{Hamiltonian}
& + \sum_{(j,k) \in \mathbb{Z}^2} V(x_{j+1,k}-x_{j,k},y_{j+1,k}-y_{j,k}) + V(y_{j,k+1}-y_{j,k},x_{j,k+1}-x_{j,k}).
\end{align}
As is discussed in \cite{FT-2002}, the model with horizontal and vertical springs may not capture all properties of elastic materials and diagonal springs are  required to describe structural stability of some materials. The KP-II approximation in the square lattices with additional diagonal springs should be possible, but computations of coefficients will become more complicated. 

Next we present the main result for the propagation of nonlinear waves in the horizontal direction. We will seek a continuous approximating function of the form 
\begin{align}\label{KPII-scaling}
x_{j,k}=\eps X\lb \xi,\eta,\tau\rb+\text{error},
\end{align}
with $\xi=\eps(j-c_1t)$, $\eta=\eps^2k$, and $\tau=\eps^3 t$. We will show that $x_{j,k}$ satisfy the equations of motion with small error controllable in $\ell^2(\mathbb{Z}^2)$ if $X\lb\xi,\eta,\tau\rb$ solves the following KP-II equation 
\begin{align}\label{KPII-formal}
2c_1\partial_\xi\partial_{\tau} X + \frac {c_1^2} {12} \partial_\xi^4 X  + 2 \alpha_1 \lb \partial_\xi X\rb \lb \partial^2_\xi X\rb  + c_2^2\partial_\eta^2 X=0.
\end{align}
For justification analysis, it is more convenient to solve the FPU system in strain variables,
\begin{align}
\label{strain-variable}
\left\{ \begin{array}{l}
u_{j,k}^{(1)} :=x_{j+1,k}-x_{j,k},\\
u_{j,k}^{(2)} :=x_{j,k+1}-x_{j,k},\\
v_{j,k}^{(1)} :=y_{j+1,k}-y_{j,k},\\
v_{j,k}^{(2)} :=y_{j,k+1}-y_{j,k}, \end{array} \right.
\end{align}
which are defined by the relative displacements between adjacent particles, 
and introduce the amplitude function $A\lb \xi,\eta,\tau\rb$ in the form
\begin{align}\label{KPII-scaling-again}
x_{j+1,k}-x_{j,k}=\eps^2 A\lb \xi,\eta,\tau\rb + \text{error}.
\end{align}
The reason for the different scaling in (\ref{KPII-scaling}) and (\ref{KPII-scaling-again}) is that we can formally consider the relationship between the function $A$ and $X$ through a Taylor expansion 
\begin{equation}
\label{correspondence}
A \lb \xi,\eta,\tau\rb = \p \xi X \lb \xi,\eta,\tau\rb + \mathcal{O}(\eps),
\end{equation} 
so that the KP-II equation \eqref{KPII-formal} can be rewritten in the form 
\begin{align}\label{KPII}
2c_1 \p \xi \p \tau A + \frac{c_1^2}{12} \p \xi^4 A + 2 \alpha_1 \p \xi \lb A \p \xi A\rb+c_2^2 \p \eta^2 A = 0,
\end{align}
Associated with a smooth solution to the KP--II equation (\ref{KPII}) at a given time $\tau$, for which both $A\in H^{s}\lb \R^2\rb$ and $\p \xi^{-1} A\in H^{s}\lb \R^2\rb$ with sufficiently large $s$, 
we define the anti-derivative by
\[
\p \xi^{-1}A (\xi,\eta) := \int_{-\infty}^\xi A(\xi',\eta) d\xi'.
\]
The solution to the KP-II equation (\ref{KPII}) is required to have enough regularity so that $\p \xi^{-1} \p \tau^2 A\in C^0\lb \lsb -\tau_0, \tau_0 \rsb,H^{s}\lb \R^2\rb\rb$ and  $\p \xi^{-2} \p \eta^2 \p \tau A\in C^0\lb \lsb -\tau_0, \tau_0 \rsb,H^{s}\lb \R^2\rb\rb$ with $s \geq 3$. Existence of such solutions is proven in Lemma \ref{welpos} (Section \ref{sec-prel}). Since $A(\eps (j-c_1t),\eps^2 k,\eps^3 t)$ 
is estimated in the $\ell^2(\Z^2)$ norm over $(j,k) \in \Z^2$, we will also need the bound on the $\ell^2(\Z^2)$ norm of these terms by their $H^s(\R^2)$ norm, this bound is obtained in Lemma \ref{lmell2} (Section \ref{sec-prel}).

The following theorem formulates the main result which justifies the KP-II approximation (\ref{KPII}) for the horizontal propagation of nonlinear waves. 

	\begin{theorem}\label{theorem-horiz}
		Let $A \in C^0\lb\lsb -\tau_0,\tau_0\rsb, H^{s+9}\lb \R^2\rb \rb$ be a solution to the KP-II equation $\eqref{KPII}$ with fixed $s \geq 0$, whose initial data 
		$A(\xi,\eta,0)=A_0$ satisfies 
		$$
		A_0\in H^{s+9}\lb \R^2\rb, \qquad 
		\p \xi ^{-2}\p \eta ^2 A_0 \in H^{s+9}\lb \R^2\rb 
		$$ 
		and 
		$$
		\p \xi^{-1}\p \eta^2\lsb \p \xi^{-2}\p \eta^2 A_0+A_0^2\rsb\in H^{s+3}\lb \R^2\rb.
		$$ 
	Then there are constants $C_0,C_1,\eps_0>0$ such that for $\eps\in (0,\eps_0)$ if the initial conditions of the two-dimensional FPU system satisfies 
	\begin{align}\label{hyp}
	\begin{split}
	&\norm{u^{(1)}_{\text {in}}-\eps^2 A_0}_{\ell^2}+\norm{u^{(2)}_{\text {in}}}_{\ell^2} +\norm{\dot{x}_{\text{in}}+\eps^2 c_1 A_0}_{\ell^2} \\
	&+\norm{v^{(1)}_{\text {in}}}_{\ell^2}+\norm{v^{(2)}_{\text {in}}}_{\ell^2}+\norm{\dot{y}_{\text {in}}}_{\ell^2}\leq C_0\eps^{\frac 5 2}
	\end{split}
	\end{align}
	then the solution to the two-dimensional FPU system satisfies
	\begin{align}\label{res}
	\begin{split}
	&\norm{u^{(1)}(t)-\eps^2 A}_{\ell^2}+\norm{u^{(2)}(t)}_{\ell^2}
	+\norm{\dot{x}(t)+ \eps^2 c_1 A}_{\ell^2} \\
	& +\norm{v^{(1)}(t)}_{\ell^2}+\norm{v^{(2)}(t)}_{\ell^2}+\norm{\dot{y}(t)}_{\ell^2}\leq C_1 \eps^{\frac 5 2},
	\end{split}
	\end{align}
	for $t\in [-\tau_0 \eps^{-3},\tau_0 \eps^{-3}]$.
\end{theorem}

\begin{remark}
	Extending this result to $\eps^{\frac7 2}$ is difficult as the next order of the asymptotic expansion has terms which are not removed by seeking solutions to the KP-II equation alone. They could be removed by seeking a function of the form 
	$$
	A\lb\xi,\eta,\tau\rb=A^{(0)}\lb\xi,\eta,\tau\rb+\eps^2 A^{(1)} \lb\xi,\eta,\tau\rb,
	$$ 
	where $A^{(0)}$ solves the KP-II equation \eqref{KPII} and $A^{(1)}$ solves an appropriately chosen linearized KP-II equation. However the linearized KP-II equation is nonhomogeneous, where the nonhomogeneous piece contains higher order antiderivative terms of $A^{(0)}$.
\end{remark}

\begin{remark}
Compared to the work \cite{Z}, the slow transverse modulations in the expansion (\ref{KPII-scaling-again}) are not due to the external small parameter 
in the potential energy $V(r,s)$. If we have the external small parameter $\eps$ as in (\ref{Hamiltonian-2D}), we can use $\eta = \eps k$ so that the comparison between the $\ell^2$ norm and the Sobolev norm would only lose $\eps^{-1}$ compared to $\eps^{-\frac 3 2}$ in Lemma \ref{lmell2}. Performing the justification analysis on a version of the FPU system as in \cite{Z} should also yield Theorem \ref{theorem-horiz} but with the errors bounds of $\eps^3$ rather than $\eps^{\frac 5 2}$.
\end{remark}

\begin{remark}
	Theorem \ref{theorem-horiz} can be proven in the setting of periodic transverse modulations, for which $A(\xi,\eta+P,\tau) = A(\xi,\eta,\tau)$ with fixed $P > 0$. In view of the scaling $\eta = \eps^2 k$, this would correspond to the periodic transverse modulations with large,  $\eps$-dependent period $P \eps^{-2}$. Similarly, one can consider 
	periodic waves of the KP-II equation satisfying $A(\xi+L,\eta,\tau) = A(\xi,\eta,\tau)$, 	for which the periodic waves of the FPU lattice 
	has large, $\eps$-dependent period $L \eps^{-1}$. The recent work 
	\cite{PP21} justifies the KP-II approximation for periodic solutions in both spatial directions as the normal form for metastable dynamics of the rectangular FPU system. 
\end{remark}

\begin{remark}
	By the obvious symmetry, the result of Theorem \ref{theorem-horiz} can be formulated for the vertical propagation of nonlinear waves in the $y$-coordinates.
\end{remark}

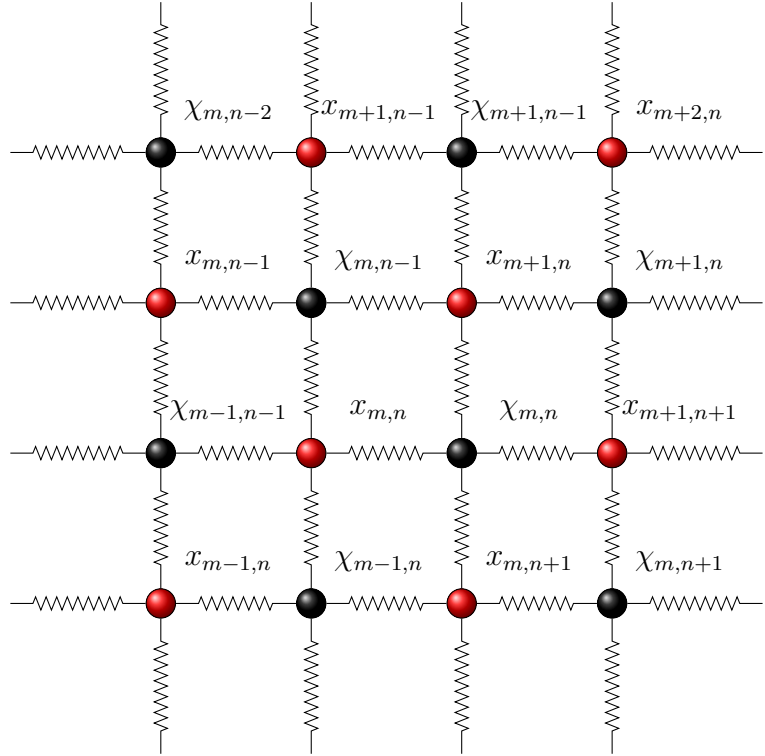
\begin{figure}[b]
	\centering
	\begin{tikzpicture}[
	wall/.style = {gray,fill=gray},
	mass1/.style = {draw,circle,ball color=red},
	mass2/.style = {draw,circle,ball color=black},
	spring/.style = {decorate,decoration={zigzag, pre length=.3cm,post length=.3cm,segment length=#1}},
	]
	\coordinate (l1) at (0,2);
	\node[mass1,label={[xshift=0.9cm, yshift=0.05cm]$x_{m-1,n}$}] (m11) at (2,2) {};
	\node[mass2,label={[xshift=0.9cm, yshift=0.05cm]$\chi_{m-1,n}$}] (m21) at (4,2) {};
	\node[mass1,label={[xshift=0.9cm, yshift=0.05cm]$x_{m,n+1}$}] (m31) at (6,2) {};
	\node[mass2,label={[xshift=0.9cm, yshift=0.05cm]$\chi_{m,n+1}$}] (m41) at (8,2) {};
	\coordinate (r1) at (10,2);
	
	\coordinate (l2) at (0,4);
	\node[mass2,label={[xshift=0.9cm, yshift=0.05cm]$\chi_{m-1,n-1}$}] (m12) at (2,4) {};
	\node[mass1,label={[xshift=0.9cm, yshift=0.05cm]$x_{m,n}$}] (m22) at (4,4) {};
	\node[mass2,label={[xshift=0.9cm, yshift=0.05cm]$\chi_{m,n }$}] (m32) at (6,4) {};
	\node[mass1,label={[xshift=0.9cm, yshift=0.05cm]$x_{m+1,n+1}$}] (m42) at (8,4) {};
	\coordinate (r2) at (10,4);
	
	\coordinate (l3) at (0,6);
	\node[mass1,label={[xshift=0.9cm, yshift=0.05cm]$x_{m,n-1}$}] (m13) at (2,6) {};
	\node[mass2,label={[xshift=0.9cm, yshift=0.05cm]$\chi_{m,n-1}$}] (m23) at (4,6) {};
	\node[mass1,label={[xshift=0.9cm, yshift=0.05cm]$x_{m+1,n}$}] (m33) at (6,6) {};
	\node[mass2,label={[xshift=0.9cm, yshift=0.05cm]$\chi_{m+1,n}$}] (m43) at (8,6) {};
	\coordinate (r3) at (10,6);

	\coordinate (l4) at (0,8);
	\node[mass2,label={[xshift=0.9cm, yshift=0.05cm]$\chi_{m,n-2}$}] (m14) at (2,8) {};
	\node[mass1,label={[xshift=0.9cm, yshift=0.05cm]$x_{m+1,n-1}$}] (m24) at (4,8) {};
	\node[mass2,label={[xshift=0.9cm, yshift=0.05cm]$\chi_{m+1,n-1}$}] (m34) at (6,8) {};
	\node[mass1,label={[xshift=0.9cm, yshift=0.05cm]$x_{m+2,n}$}] (m44) at (8,8) {};
	\coordinate (r4) at (10,8);
	
	%\draw [->] (m22) -- (m13);
	
	%horizontal
	\draw[spring=4pt] (l1) -- node[above] {} (m11);
	\draw[spring=4pt] (m11) -- node[above] {} (m21);
	\draw[spring=4pt] (m21) -- node[above] {} (m31);
	\draw[spring=4pt] (m31) -- node[above] {} (m41);
	\draw[spring=4pt] (m41) -- node[above] {} (r1);
	
	\draw[spring=4pt] (l2) -- node[above] {} (m12);
	\draw[spring=4pt] (m12) -- node[above] {} (m22);
	\draw[spring=4pt] (m22) -- node[above] {} (m32);
	\draw[spring=4pt] (m32) -- node[above] {} (m42);
	\draw[spring=4pt] (m42) -- node[above] {} (r2);
	
	\draw[spring=4pt] (l3) -- node[above] {} (m13);
	\draw[spring=4pt] (m13) -- node[above] {} (m23);
	\draw[spring=4pt] (m23) -- node[above] {} (m33);
	\draw[spring=4pt] (m33) -- node[above] {} (m43);
	\draw[spring=4pt] (m43) -- node[above] {} (r3);
	
	\draw[spring=4pt] (l4) -- node[above] {} (m14);
	\draw[spring=4pt] (m14) -- node[above] {} (m24);
	\draw[spring=4pt] (m24) -- node[above] {} (m34);
	\draw[spring=4pt] (m34) -- node[above] {} (m44);
	\draw[spring=4pt] (m44) -- node[above] {} (r4);
	
	%vertical
	
	\draw[spring=4pt] (2,10) -- node[above] {} (m14);
	\draw[spring=4pt] (m14) -- node[above] {} (m13);
	\draw[spring=4pt] (m13) -- node[above] {} (m12);
	\draw[spring=4pt] (m12) -- node[above] {} (m11);
	\draw[spring=4pt] (m11) -- node[above] {} (2,0);
	
	\draw[spring=4pt] (4,10) -- node[above] {} (m24);
	\draw[spring=4pt] (m24) -- node[above] {} (m23);
	\draw[spring=4pt] (m23) -- node[above] {} (m22);
	\draw[spring=4pt] (m22) -- node[above] {} (m21);
	\draw[spring=4pt] (m21) -- node[above] {} (4,0);
	
	\draw[spring=4pt] (6,10) -- node[above] {} (m34);
	\draw[spring=4pt] (m34) -- node[above] {} (m33);
	\draw[spring=4pt] (m33) -- node[above] {} (m32);
	\draw[spring=4pt] (m32) -- node[above] {} (m31);
	\draw[spring=4pt] (m31) -- node[above] {} (6,0);
	
	\draw[spring=4pt] (8,10) -- node[above] {} (m44);
	\draw[spring=4pt] (m44) -- node[above] {} (m43);
	\draw[spring=4pt] (m43) -- node[above] {} (m42);
	\draw[spring=4pt] (m42) -- node[above] {} (m41);
	\draw[spring=4pt] (m41) -- node[above] {} (8,0);	
	\end{tikzpicture}
	\caption{A diatomic mass spring system arranged in a square lattice}\label{fig:2D-diag}
\end{figure}

Finally, we consider the diagonal propagation of nonlinear waves in the 
FPU lattice shown on Figure \ref{fig:2Dlat} and introduce a new coordinate system by 
$$
m=\frac{j+k}2, \quad n=\frac{j-k}{2}, \quad (j,k) \in \mathbb{Z}^2.
$$
In the new coordinate system the particle experiences nearest-neighbour interactions with neighbours located a half lattice site away. Due to this we redefine $x_{j,k}$ as $x_{m,n}$ and introduce $\chi_{m,n} :=x_{m+\frac 1 2, n+\frac 1 2}$. The FPU system becomes a diatomic system where $x_{m,n}$ particles communicate with four $\chi_{m,n}$ nearest-neighbour particles and vice versa, see Figure \ref{fig:2D-diag} for an illustration. 

We will seek a continuous approximating function of the form 
$$
x_{m,n}=\eps X\lb \eps(m-c^*_1t),\eps^2(n-c_2^* t),\eps^3 t\rb+\text{error},
$$ 
where $c_1^*=\frac{1}{2} \sqrt{c_1^2+c_2^2}$ and $c_2^*=\frac{1}{2} \sqrt{c_1^2-c_2^2}$. It is hard to control the error in a general case 
because nonlocal terms related to the solution of the KP-II equation appear in lower orders of the asymptotic approximation. However, if $c_2 = c_1$ and 
$\alpha_2 = 2 \alpha_1$, then $c_1^* = \frac{c_1}{\sqrt{2}}$, $c_2^* = 0$ and furthermore, the FPU system is satisfied by the invariant reduction $x_{j,k}=y_{j,k}$. By using the strain variables, 
\begin{align}\label{KPII-scaling-diag}
x_{m+1,n}-x_{m,n}=\eps^2 A\lb \eps(m-c^*_1t),\eps^2 n,\eps^3 t\rb+\text{error},
\end{align}
which corresponds to the displacement along the main diagonal,
we derive the following KP-II equation for $A(\xi,\eta,\tau)$ 
with the same relation (\ref{correspondence}) between $A$ and $\p \xi X$:
\begin{align}\label{KPIIdiag}
2c_1^* \p \xi \p\tau A + \frac{(c_1^*)^2}{48} \p \xi^4 A + 
\alpha_1 \p \xi \lb A \p \xi A \rb + (c_1^*)^2 \p \eta^2 A 
=0.
\end{align}
where $\xi = \eps (m - c_1^* t)$, $\eta = \eps^2 n$, and $\tau = \eps^3 t$. 
Similarly to the case of horizontal propagation, we can redefine 
the strain variables for the diagonal propagation as follows:
\begin{align}
\label{strain-variabl-diag}
\left\{ \begin{array}{l}
u_{m,n}^{(1)} :=x_{m+1,n}-x_{m,n},\\
u_{m,n}^{(2)} :=x_{m,n+1}-x_{m,n},\\
v_{m,n}^{(1)} :=\chi_{m+1,n}-\chi_{m,n},\\
v_{m,n}^{(2)} :=\chi_{m,n+1}-\chi_{m,n}.
\end{array} \right.
\end{align}

The following theorem is similar to Theorem \ref{theorem-horiz} 
but it justifies the KP-II approximation (\ref{KPIIdiag}) for the diagonal propagation of nonlinear waves.  

\begin{theorem}\label{theorem-diag} 
	Let $A \in C^0\lb\lsb -\tau_0,\tau_0\rsb, H^{s+9}\lb \R^2\rb \rb$ be a solution to the KP-II equation $\eqref{KPIIdiag}$ with fixed $s \geq 0$, whose initial data $A(\xi,\eta,0)=A_0$ satisfies 
	$$
	A_0 \in H^{s+9}\lb \R^2\rb, \qquad 
	\p \xi ^{-2}\p \eta ^2 A_0 \in H^{s+9}\lb \R^2\rb 
	$$ 
	and 
	$$
	\p \xi^{-1}\p \eta^2\lsb \p \xi^{-2}\p \eta^2 A_0+A_0^2\rsb\in H^{s+3}\lb \R^2\rb.
	$$ 
	Then there are constants $C_0,C_1,\eps_0>0$ such that for $\eps\in (0,\eps_0)$ if the initial conditions of the two-dimensional FPU system
	with $c_2 = c_1$ and $\alpha_2 = 2 \alpha_1$ satisfies 
	\begin{align}\label{hyp-diag}
\begin{split}
&\norm{u^{(1)}_{\text {in}}-\eps^2 A_0}_{\ell^2}+\norm{u^{(2)}_{\text {in}}}_{\ell^2} +\norm{\dot{x}_{\text{in}}+ \eps^2 c_1^* A_0}_{\ell^2} \\
&+\norm{v^{(1)}_{\text {in}}-\eps^2 A_0}_{\ell^2}+\norm{v^{(2)}_{\text {in}}}_{\ell^2}+\norm{\dot{\chi}_{\text {in}}+ \eps^2 c_1^* A_0}_{\ell^2}\leq C_0\eps^{\frac 5 2}
\end{split}
\end{align}
	then the solution to the two-dimensional FPU system satisfies
	\begin{align}\label{res-diag}
\begin{split}
&\norm{u^{(1)}(t)-\eps^2 A}_{\ell^2}+\norm{u^{(2)}(t)}_{\ell^2}
+\norm{\dot{x}(t) +\eps^2 c_1^* A}_{\ell^2} \\
& +\norm{v^{(1)}(t)-\eps^2 A}_{\ell^2} +\norm{v^{(2)}(t)}_{\ell^2} +\norm{\dot{\chi}(t) + \eps^2 c_1^* A}_{\ell^2}\leq C_1 \eps^{\frac 5 2},
\end{split}
\end{align}
	for $t\in [-\tau_0 \eps^{-3},\tau_0 \eps^{-3}]$.
\end{theorem}
  
\begin{remark}
	It is an open problem to justify the KP-II approximation for diagonal propagation with $c_2 \neq c_1$ and $\alpha_2 \neq 2 \alpha_1$ or for other directions along the lattice. The main challenge arises in controlling in Sobolev norm of the nonlocal terms computed at solutions of the KP-II equation. Even for Theorem \ref{theorem-diag}, 
	if we use another equivalent choice for the asymptotic approximation, e.g. \begin{align}
	\chi_{m,n} - x_{m,n} = \frac{1}{2} \eps^2 A\lb \eps(m-c^*_1t),\eps^2 n,\eps^3 t\rb+\text{error},
	\end{align}  
	 the asymptotic expansions contain some non-local terms. Although these nonlocal terms can be transformed away by near-identity transformations, the choice of \eqref{KPII-scaling-diag} allows us to avoid the nonlocal terms, see Section 4.1.
\end{remark}
  
The remainder of the paper is organized as follows. 
Section \ref{sec-prel} contains preliminary results needed for the justification 
analysis. In particular,  we extend \cite{GS-KP} to obtain additional estimates on solutions of the KP-II equation and extend \cite{DP,SW1} to control $\ell^2(\mathbb{Z}^2)$ norm at the slowly varying solution of the KP-II equation. Section \ref{sec-proof-1} gives the proof of Theorem \ref{theorem-horiz} after equation of motion are set up in the strain variables  and the near-identity transformations are performed to reduce the residual.
The error terms are controlled from the energy estimates and Gronwall's inequality. Section \ref{sec-proof-2} gives relevant details 
for the very similar proof of Theorem \ref{theorem-diag}. 
Summary and discussion of further questions are contained in 
the concluding Section \ref{sec-conclusion}.

\section{Preliminary results}
\label{sec-prel}

Consider the Cauchy problem for the normalized KP-II equation
\begin{equation}
\label{KPIInorm} 
\left\{ \begin{array}{l} \p \tau A + \p \xi \lb A^2\rb+\p \xi^3 A+\p \xi^{-1}\p \eta^2 A=0, \quad t > 0, \\
A |_{t=0} = A_0. \end{array} \right.
\end{equation} 
The normalized KP-II equation differs slightly from \eqref{KPII} and \eqref{KPIIdiag}, in the choice of constants. However, the constants in the KP-II equation can be changed, as long as each constant is positive, through a scaling of its variables.

Global well-posedness of the Cauchy problem (\ref{KPIInorm}) was established 
in $H^s\lb \mathbb T^2\rb$ or $H^s\lb \R^2\rb$ with any $s\geq 0$ in \cite{Bourgain1993}, provided that the initial data satisfies 
the constraint 
\begin{equation}
\label{zero-mass}
\int_\R A_0 \lb \xi, \eta\rb d\xi = 0, \quad \mbox{\rm for every  } \eta. 
\end{equation}
The result was proven by combining local well-posedness and conservation laws, namely conservation of the $L^2$ norm along the solution.

Local well-posedness result was extended in \cite{Tzvetkov1999} to Sobolev spaces of the type $H^{s_1,s_2}\lb \R^2 \rb$, with $s_1>-\frac 1 4$ and $s_2\geq 0$, the global result was also obtained by the conservation laws provided that $s_1\geq 0$. It was further shown in \cite{Takaoka2000} that the zero mean constraint (\ref{zero-mass}) can be dropped in the local-wellposedness result. 

We will use the following local well-posedness result  \cite{GS-KP} (see \cite{Ukai1989} for earlier work).

\begin{proposition}\cite{GS-KP}
	\label{KPIIwellpos}
For any $A_0\in H^{s+6}\lb \R^2\rb$ such that 
$\p \xi^{-2}\p \eta^2 A_0\in H^{s+3}\lb \R^2\rb$ with fixed $s \geq 0$, there exists $\tau_0>0$ such that the Cauchy problem \eqref{KPIInorm} admits a unique solution
\begin{align*}
A \in \solsp 0 {\tau_0}{s+6}\cap\solsp 1 {\tau_0}{s+3}\cap \solsp 2 {\tau_0}{s}
\end{align*}
such that $\p \xi^{-1}\p \eta A \in \solsp 0 {\tau_0}{s+5} \cap \solsp 1 {\tau_0}{s+2}$.
\end{proposition}

However, for our work, we need to extend this result to  $C^3\lb\lsb-\tau_0,\tau_0\rsb,H^s\rb$
 since we need the antiderivative term $\p \xi^{-1} A$ as a function of $\tau$ to be twice 
continuously differentiable in some Sobolev space.
The following lemma presents the corresponding extension of Proposition \ref{KPIIwellpos}.

	\begin{lemma}\label{welpos}
For any $A_0\in H^{s+9}\lb \R^2\rb$ such that 
$\p \xi ^{-2}\p \eta ^2 A_0 \in H^{s+6}\lb \R^2\rb $ and 
$$
\p \xi^{-1}\p \eta^2\lsb \p \xi^{-2}\p \eta^2 A_0 +A_0^2\rsb\in H^{s+3}\lb \R^2\rb
$$ 
with fixed $s \geq 0$, there exists $\tau_0>0$ such that 
the Cauchy problem (\ref{KPIInorm}) admits a unique solution 
$$
A\in \solsp {0}{\tau_0}{s+9}\cap \solsp {1}{\tau_0}{s+6} \cap \solsp {2}{\tau_0}{s+3}\cap \solsp {3}{\tau_0}{s}
$$
such that 
$$
\p \xi^{-1}\p \eta A \in \solsp 0 {\tau_0}{s+8} \cap \solsp 1 {\tau_0}{s+5} 
\cap \solsp 2 {\tau_0}{s+2}
$$
and  
$$
\p \xi^{-2}\p \eta^2  A \in \solsp {0}{\tau_0}{s+6} \cap \solsp 1 {\tau_0}{s+3}.
$$
\end{lemma}

\begin{proof}	
Assume that $A$ solves the Cauchy problem \eqref{KPIInorm} and set 
$D := \p \xi^{-2}\p \eta ^2 A$. By Proposition \ref{KPIIwellpos}, since the initial data satisfies $A_0\in H^{s+9}\lb \R^2\rb$ and $\p \xi ^{-2}\p \eta ^2 A_0 \in H^{s+6}\lb \R^2\rb $ with $s \geq 0$, the KP-II equation has a solution \begin{equation}
\label{sol-kp}
A\in \solsp {0}{\tau_0}{s+9}\cap \solsp {1}{\tau_0}{s+6} \cap \solsp {2}{\tau_0}{s+3}
\end{equation}
such that $\p \xi^{-1}\p \eta  A \in \solsp 0 {\tau_0}{s+8} \cap \solsp 1 {\tau_0}{s+5}$. Taking $\p \eta^2 \lb\cdot\rb$ of the KP-II equation yields
	\[
	\p \xi^2\p \tau D+\p \xi\p \eta^2 \lb A^2\rb+\p \xi^5 D+\p \xi\p \eta^2 D = 0.
	\]
Setting $\tilde D := D + A^2$ and taking $\p \xi^{-2}\lb \cdot\rb $ yields the evolution equation
\begin{equation}
\label{eq-D-tilde}
	\p \tau \tilde D + \p \xi^3 \tilde D + \p \xi^{-1}\p \eta^2 \tilde D = \p \tau \lb A^2\rb + \p \xi^3 \lb A^2 \rb.
\end{equation}

Let us define $\Omega :=  \p \xi^{-1}\p \eta^2+\p \xi^3$ and $S(\tau) :=e^{\tau\Omega}$. Since $\Omega$ is skew-adjoint, the evolution operator $S\lb \tau\rb$ is unitary in $L^2(\R^2)$. Using Duhamel's principle we can write the evolution equation (\ref{eq-D-tilde}) in the integral form:
	\[
	\tilde D\lb\tau\rb=S(\tau)\tilde D_0 +\int_0^\tau S(\tau-s) \left[ \p s \lb A^2\rb + \p \xi^3 \lb A^2 \rb \right](s) ds,
	\]
	where $\tilde{D}_0 = \p \xi^{-2}\p \eta ^2 A_0 + A_0^2$.
	Expressing back $D=\tilde D -A$ gives us an integral equation for $D(\tau)$ of the form
\begin{equation}
\label{eq-D-exact}
	D\lb\tau\rb=S(\tau)\tilde D_0 -A(\tau)^2+\int_0^\tau S(\tau-s) \left[ \p s \lb A^2\rb + \p \xi^3 \lb A^2 \rb \right](s) ds.
\end{equation}
Since $\tilde{D}_0 \in H^{s+6}(\R^2)$ and $S(\tau)$ is unitary in $L^2(\R^2)$, 
the local solution (\ref{sol-kp}) yields $D \in \solsp 0 {\tau_0}{s+6}$.

Taking the time derivative of the integral equation (\ref{eq-D-exact}), using $$
\frac d {d\tau} S(\tau-s)=-\frac d {ds}S(\tau-s),
$$ 
and integrating by parts, we obtain
	\begin{align*}
	\p \tau D(\tau)=& S(\tau) \left[ \Omega\tilde D_0 + 2A_0 \p \tau A_0 + \p \xi^3 \lb A^2_0\rb \right] -2A(\tau) \p \tau A(\tau) \\
	&+ \int_0^\tau S\lb \tau-s\rb  \p s \left[ \p s \lb A^2\rb+ \p \xi^3 \lb A^2 \rb \right](s)  ds.
	\end{align*}
	Since $\Omega \tilde D_0 \in H^{s+3}(\R^2)$, then the solution (\ref{sol-kp}) satisfies $\p \tau D\in \solsp 0 {\tau_0} {s+3}$, which gives $\p \xi^{-2}\p \eta^2 A \in \solsp {1}{\tau_0}{s+3}$. 
	
	It remains to control $\p \tau^3 A$ and $\p \xi^{-1}\p \eta \p \tau ^2A$, 
	which is achieved by computing the time derivatives of the KP-II equation:
	\begin{align*}
  \p \tau ^2A &=-\lb 2\p \xi \lb A \p \tau A \rb+\p \xi^3 \p \tau A + \p \xi^{-1}\p \eta^2 \p \tau A \rb, \\
	\p \tau ^3A&=- \lb 2\p \xi \lb \lb \p \tau A \rb^2 + A \p \tau^2 A \rb+\p \xi^3 \p \tau^2 A +\p \xi^{-1}\p \eta^2 \p \tau^2 A  \rb
	\end{align*}
It follows from (\ref{sol-kp}) that all but the last term in $\p \tau ^3A$ are in $ C^{0}\lb \lsb -\tau_0, \tau_0\rsb, H^{s}\rb$. By using the expression for $\p \tau^2 A$, we check that 
	\begin{align}
		\label{A-der}
	\p \xi^{-1} \p \tau ^2A = -\lb 2 A\p \tau A +\p \xi^2 \p \tau A+ \p \tau D \rb,
	\end{align}
so that $\p \xi^{-1} \p \tau^2 A \in C^{0}\lb \lsb -\tau_0, \tau_0\rsb, H^{s+3}\rb$, which yields $A \in \solsp {3}{\tau_0}{s}$ and $\p \xi^{-1}\p \eta A \in \solsp 2 {\tau_0}{s+2}$.
\end{proof}

Next we derive a useful bound on the $\ell^2(\mathbb{Z}^2)$ norm of a function  expressed in terms of the slowly varying solution of the KP-II equation defined in $H^s(\R^2)$. A similar result for one-dimensional chains was obtained in \cite{DP} (see also \cite{SW1} for earlier work). 

\begin{proposition}\cite{DP}
	Let $u_{j}=U(\eps j)$, where $U\in H^1(\R)$. There is a constant $C>0$ such that for every $\eps\in (0,1]$ we have 
	\begin{equation}
	\label{bound-DP}
	\norm u_{\ell^2(\Z)}\leq  C \eps^{-\frac 1 2} \norm U_{H^1(\R)},\qquad \forall U\in H^1\lb \R\rb.
	\end{equation}
	\label{prop-mell}
\end{proposition}

The following lemma generalizes Proposition \ref{prop-mell} 
for two-dimensional square lattices.

\begin{lemma}\label{lmell2}
	Let $u_{j,k}=U(\eps j,\eps^2 k)$, where $U\in H^s(\R^2)$ with fixed $s>1$. There is a constant $C_s>0$ that depends on $s$ such that for every $\eps\in (0,1]$ we have 
	\begin{equation}
	\label{mell2}
	\norm u_{\ell^2(\Z^2)}\leq  C_s \eps^{-\frac 3 2} \norm U_{H^s(\R^2)},\qquad \forall U\in H^{s}\lb \R^2\rb.
	\end{equation}
\end{lemma}

\begin{proof}
We use the discrete Fourier transform defined by 
	\begin{equation*}
	\uh \theta \phi = \ds_{(j,k)\in \Z^2} u_{j,k} e^{-i(j\theta + k\phi)}
		\end{equation*}
with the inverse transform given by	
	\begin{equation}\label{FT1}
	u_{j,k} = \frac{1}{(2\pi)^2}\int_{-\pi}^\pi \int_{-\pi}^\pi \uh \theta \phi e^{i(j\theta + k \phi)} d\theta d\phi.
	\end{equation}
Since $U\in H^s(\R^2)$ and $s>1$, the Fourier transform of $U$ denoted by $\hat{U}$ can be defined through the standard formula. 
We represent 
	\begin{align}\label{FT2}
	\begin{split}
	u_{j,k}&=U(\eps j, \eps^2 k) = \frac{1}{(2\pi)^2} \rint \rint \hat U(\tilde p, \tilde q)e^{i(\eps j\tilde p+\eps^2 k\tilde q)}d\tilde pd\tilde q\\
	&=\frac{1}{(2\pi)^2\eps^3}\rint \rint \hat U \left(\frac p \eps, \frac q {\eps^2}\right)e^{i(jp+kq)}dpdq\\
	&=\frac{1}{(2\pi)^2\eps^3} \sum_{(n,m)\in \Z^2}\int_{(2n-1)\pi}^{(2n+1)\pi}\int_{(2m-1)\pi}^{(2m+1)\pi}\hat U \left(\frac p \eps, \frac q {\eps^2}\right)e^{i(jp+kq)}dpdq\\
	&=\frac{1}{(2\pi)^2\eps^3} \sum_{(n,m)\in \Z^2} \sint \pi \sint \pi\hat U \left(\frac {\theta+2\pi m} \eps, \frac {\phi +2\pi n} {\eps^2}\right)e^{i(j\theta+k\phi)}d\theta d\phi.
	\end{split}
	\end{align}
For any finite subset $\Lambda \subset \Z^2$ we have 
	\begin{align*}
	\sum_{(n,m)\in \Lambda} \sint \pi \sint \pi \left|\hat U \left(\frac {\theta+2\pi m} \eps, \frac {\phi +2\pi n} {\eps^2}\right)\right|d\theta d\phi \leq \rint \rint \left|\hat U \left(\frac p \eps, \frac q {\eps^2}\right)\right|dp dq,
	\end{align*}
	hence 
	\[
	\ds_{(n,m)\in \Z^2} \sint \pi \sint \pi\left|\hat U \left(\frac {\theta+2\pi m} \eps, \frac {\phi +2\pi n} {\eps^2}\right)\right|d\theta d\phi
	\leq \| \hat U \|_{L^1(\R^2)},
	\] 
	where $\| U \|_{L^1(\R^2)} < \infty$ if $U \in H^s(\R^2)$ with $s > 1$.
	Then we can interchange summation and integration in \eqref{FT2} by the Fubini-Tonelli theorem. Comparing \eqref{FT1} and \eqref{FT2} yields
	\begin{align*}
	\uh \theta \phi=\frac 1 {\eps^3}\sum_{(n,m)\in \Z^2} \hat U \left(\frac {\theta+2\pi m} \eps, \frac {\phi +2\pi n} {\eps^2}\right).
	\end{align*}
Parseval's identity yields
	\begin{align*}
	\|u\|_{\ell^2(\Z^2)}^2 &= \frac{1}{(2\pi)^2\eps^6}\sint \pi \sint \pi \left| \sum_{(n,m)\in \Z^2}\hat U \left(\frac {\theta+2\pi m} \eps, \frac {\phi +2\pi n} {\eps^2}\right) \right|^2 d\theta d\phi\\
	&\leq
	\frac{1}{(2\pi)^2\eps^6}  \sum_{\substack{ (n_1,m_1)\in \Z^2\\(n_2,m_2)\in \Z^2}} \sint \pi \sint \pi \left|\hat U \left(\frac {\theta+2\pi m_1} \eps, \frac {\phi +2\pi n_1} {\eps^2}\right) \right| \\
	&\phantom{\leq\frac{1}{(2\pi)^2\eps^6}\sint \pi \sint \pi  \sum_{\substack{ (n_1,m_1)\in \Z^2\\(n_2,m_2)\in \Z^2}}}\times\left|\hat U \left(\frac {\theta+2\pi m_2} \eps, \frac {\phi +2\pi n_2} {\eps^2}\right) \right| d\theta d\phi.
	\end{align*}
	Denote $\jb {x,y}_2=\sqrt {1+x^2+y^2}$. Inserting the weights $\jb{\frac{\pi m_1}{\eps},\frac{\pi n_1}{\eps^2}}^{-2s}_2$ and $\jb{\frac{\pi m_2}{\eps},\frac{\pi n_2}{\eps^2}}^{-2s}_2$, then applying Young's inequality, $ab\leq \frac 1 2 a^2 +\frac 1 2 b^2$, yields
	\begin{align*}
	\|u\|_{\ell^2(\Z^2)}^2 
	 & \leq \frac{1}{(2\pi)^2\eps^6}\sum_{\substack{ (n_1,m_1)\in \Z^2\\(n_2,m_2)\in \Z^2}} \jb{\frac{\pi m_1}{\eps},\frac{\pi n_1}{\eps^2}}^{-2s}_2\jb{\frac{\pi m_2}{\eps},\frac{\pi n_2}{\eps^2}}^{-2s}_2\\
	&\qquad \times\left( 	\sint \pi \sint \pi \right. \frac 1 2\jb{\frac{\pi m_1}{\eps},\frac{\pi n_1}{\eps^2}}^{4s}_2\left|\hat U \left(\frac {\theta+2\pi m_1} \eps, \frac {\phi +2\pi n_1} {\eps^2}\right) \right|^2 d\theta d\phi \\
	&\qquad \qquad +\left. \sint \pi \sint \pi\frac 1 2\jb{\frac{\pi m_2}{\eps},\frac{\pi n_2}{\eps^2}}^{4s}_2 \left|\hat U \left(\frac {\theta+2\pi m_2} \eps, \frac {\phi +2\pi n_2} {\eps^2}\right) \right|^2 d\theta d\phi  \right).
	\end{align*}
	Hence, by symmetry of coefficients, we obtain:
{\small	\begin{align*}
	\|u\|_{\ell^2(\Z^2)}^2 \leq \frac{1}{(2\pi)^2\eps^6}&\left(\sum_{(n_1,m_1)\in \Z^2} \jb{\frac{\pi m_1}{\eps},\frac{\pi n_1}{\eps^2}}^{-2s}_2\right)\\\times&\left(\sum_{(n_2,m_2)\in \Z^2} \sint \pi \sint \pi\jb{\frac{\pi m_2}{\eps},\frac{\pi n_2}{\eps^2}}^{2s}_2 \left|\hat U \left(\frac {\theta+2\pi m_2} \eps, \frac {\phi +2\pi n_2} {\eps^2}\right) \right|^2 d\theta d\phi\right).
	\end{align*}}
	When $\eps\in (0,1]$ the double series in first term in $\norm{u}^2_{\ell^2(\Z)}$ converges for $s>1$ by the integral test, hence $\exists C_s>0$ so that 
	$$
	\sum_{(n_1,m_1)\in \Z^2} \jb{\frac{\pi m_1}{\eps},\frac{\pi n_1}{\eps^2}}^{-2s}_2<C_s^2.
	$$	
The second term in $\norm{u}^2_{\ell^2(\Z)}$ is related to the $H^s$ norm of $U$ given by,	
	\begin{align*}
	\|U\|_{H^s}^2 &= \frac{1}{(2\pi)^2}\rint \rint \jb{\tilde p,\tilde q}^{2s}_2\left|\hat U(\tilde p, \tilde q)\right|^2d\tilde pd\tilde q\\
	&= \frac{1}{(2\pi)^2\eps^3}\sum_{(n_2,m_2)\in \Z^2}\sint \pi \sint \pi\jb{\frac{\theta+2\pi m_2}{\eps},\frac{\phi+2\pi n_2}{\eps^2}}^{2s}_2 \\
	&\phantom{= \frac{1}{(2\pi)^2\eps^3}\sum_{(n_2,m_2)\in \Z^2}\sint \pi \sint \pi}\times\left|\hat U \left(\frac {\theta+2\pi m_2} \eps, \frac {\phi +2\pi n_2} {\eps^2}\right) \right|^2 d\theta d\phi.
	\end{align*}
Since 
$$
\jb{\frac{\pi m_2}{\eps},\frac{\pi n_2}{\eps^2}}^{2s}_2 \leq
\jb{\frac{\theta+2\pi m_2}{\eps},\frac{\phi+2\pi n_2}{\eps^2}}^{2s}_2, 
\quad \forall \theta,\phi \in [-\pi,\pi],
$$ 
the second term in $\norm{u}^2_{\ell^2(\Z)}$ is bounded above by $(2\pi)^2\eps^3 \| U \|^2_{H^s}$. Hence, we obtain 
$\|u\|_{\ell^2(\Z^2)}^2 \leq \eps^{-3}C_s^2 \| U \|^2_{H^s}$ which yields 
(\ref{mell2}).
\end{proof}

\section{Proof of Theorem \ref{theorem-horiz}}
\label{sec-proof-1}

Here we use the results of Lemmas \ref{welpos} and \ref{lmell2} in order to prove Theorem \ref{theorem-horiz}. We start by writing equations of motions in terms of the strain variables (\ref{strain-variable}):
	\begin{equation}\label{eom}
\begin{alignedat}{1}
\dot u^{(1)}_{j,k}=&w_{j+1,k}-w_{j,k},\\
\dot u^{(2)}_{j,k}=&w_{j,k+1}-w_{j,k},\\
\dot v^{(1)}_{j,k}=&z_{j+1,k}-z_{j,k},\\
\dot v^{(2)}_{j,k}=&z_{j,k+1}-z_{j,k},\\
\dot w_{j,k} = &c_1^2\lb\ux j k - \ux {j-1} k\rb + c_2^2\lb\uy j k - \uy {j}{k-1}\rb\\
&+\alpha_1 \left[ \lb \ux j k\rb^2-\lb \ux{j-1}k\rb^2\right]\\
&+\alpha_2 \left[\uy j k \vy j k - \uy j{k-1}\vy j {k-1}+\frac 1 2 \lb\vx j k\rb^2 - \frac 1 2 \lb \vx{j-1}k\rb^2\right]\\
\dot z_{j,k} =&c_1^2\lb\vy j k - \vy {j} {k-1}\rb + c_2^2\lb\vx j k - \vx {j-1}{k}\rb, \\
&+\alpha_1 \left[ \lb \vy j k\rb^2-\lb \vy{j}{k-1}\rb^2\right]\\
&+\alpha_2 \left[\ux j k \vx j k - \ux {j-1}{k}\vx {j-1} {k}+\frac 1 2 \lb\uy j k\rb^2 - \frac 1 2 \lb \uy{j}{k-1}\rb^2\right],
\end{alignedat}
\end{equation}
where $w_{j,k} := \dot{x}_{j,k}$, $z_{j,k} := \dot{y}_{j,k}$, and $(j,k) \in \Z^2$. The justification procedure is divided into the following four steps. 

\subsection{Step 1. Decomposition} 

Let us use the following decomposition,
\begin{equation}\label{eq:decom}
\begin{alignedat}{2}
\ux j k &= \eps^2\At + && \eps^2\Ux j k\\
\uy j k &=  \eps^2U_\eps\lb \xi, \eta, \tau\rb + && \eps^2\Uy j k\\
\vx j k &=  && \eps^2\Vx j k\\
\vy j k &=  && \eps^2\Vy j k\\
w_{j,k}&=  \eps^2 W_\eps\lb \xi,\eta,\tau\rb + && \eps^2W_{j,k}\\
z_{j,k}&= && \eps^2Z_{j,k}
\end{alignedat}
\end{equation}
where $\xi = \eps\lb j-c_1t\rb$, $\eta =\eps^2  k$, and $\tau=\eps^3 t$. The leading-order function $A$ is defined as a suitable solution to the KP-II equation \eqref{KPII}, whereas the $\eps$-dependent functions $U_\eps$ and $W_\eps$ are introduced to eliminate the lower order terms in $\eps$ arising from time derivatives and finite differences of $A$ in the first and second equations of system (\ref{eom}).

We denote the error terms of the formal power $\eps^5$ by $\mathcal{O}(\eps^5)$ and define $W_{\eps}$ and $U_{\eps}$ from the following equations:
\begin{align}\label{WxR}
W_\eps\lb \xi + \eps, \eta \rb -W_\eps\lb \xi , \eta \rb = -\eps c_1 \p\xi A(\xi,\eta) + \eps^3 \p \tau A(\xi,\eta) + \mathcal{O}(\eps^5)
\end{align}
and 
\begin{align} \label{Ut}
\ye {W_\eps} + - W_{\eps}\lb \xi , \eta \rb = -\eps c_1 \p\xi  {U_\eps}(\xi,\eta) + \eps^3 \p \tau {U_\eps}(\xi,\eta) +  \mathcal{O}(\eps^5),
\end{align}
where the time dependence is dropped from the list of arguments. 

We look for an approximate solution to \eqref{WxR} in the form:
\begin{equation}\label{eq:Weps}
W_\eps=\We 0 +\eps \We 1 + \eps^2 \We 2+\eps^3 \We 3,
\end{equation}
where the functions $W^{(j)}$ depend on $(\xi,\eta)$ and decay to zero at infinity. Plugging  \eqref{eq:Weps} into \eqref{WxR} and expanding each $W^{(j)}$ in Taylor series, we get
\begin{align*}
&\eps \p\xi\We 0 + \frac {1}{2} \eps^2\p \xi ^2\We 0 + \frac {1}{6} \eps^3 \p \xi ^3\We 0+ \frac {1}{24} \eps^4 \p \xi ^4\We 0%+\eps^5\frac{1}{5!} \p \xi^5 \We 0
\\& \qquad +\eps^2\p \xi\We 1 + \frac {1}{2} \eps^3\p \xi ^2\We 1 + \frac {1}{6} \eps^4 \p \xi ^3\We 1 +\eps^3\p \xi \We 2 + \frac {1}{2} \eps^4 \p \xi ^2\We 2
\\& \qquad  \qquad 
+\eps^4\p \xi\We 3 
=-\eps c_1 \p\xi A + \eps^3 \p \tau A +  \mathcal{O}(\eps^5).
\end{align*}
Grouping terms by their orders in powers of $\eps$ yields a sequence of equations with their relevant solutions:
\begin{align}\label{eq:Wesol}
\begin{split}
& \mathcal{O}(\eps) : \quad \p\xi\We 0 = -c_1 \p \xi A\\
&\phantom{O(\eps): \p\xi\We 0 } \qquad \implies \We 0 = -c_1 A\\
& \mathcal{O}(\eps^2) : \quad \frac {1}2 \p \xi ^2\We 0+\p \xi \We 1 = 0 \\
&\phantom{O(\eps): \p\xi\We 0 } \qquad  \implies \We 1 = \frac {c_1}2 \p \xi A\\
& \mathcal{O}(\eps^3) : \quad \frac 1 6 \p \xi ^3 \We 0 + \frac 1 2 \p \xi ^2 \We 1 + \p \xi \We 2=\p \tau A  \\
&\phantom{O(\eps): \p\xi\We 0 } \qquad  \implies \We 2 = \p \xi^{-1}\p \tau A -\frac {c_1} {12}\p \xi^2 A\\
&  \mathcal{O}(\eps^4) : \quad \frac 1 {24}  \p \xi ^4\We 0 + \frac 1 6 \p \xi ^3 \We 1 + \frac 1 2\p \xi^2 \We 2+\p \xi \We 3=0 \\
&\phantom{O(\eps): \p\xi\We 0 } \qquad  \implies \We 3 = -\frac 1 2 \p \tau A.  \\
\end{split}
\end{align}
With the choice in \eqref{eq:Wesol}, this construction ensures that equation  \eqref{WxR} is satisfied up to and including the order of $\mathcal{O}(\eps^4)$. Substituting \eqref{eq:Wesol} into \eqref{eq:Weps} yields
	\begin{equation}\label{We}
W_\eps= -c_1 A+\eps \lb \frac{c_1}{2}\p \xi A \rb + \eps^2 \lb\p \xi^{-1} \p \tau A-\frac{c_1}{12} \p\xi^2 A\rb-\eps^3 \lb \frac 1 2 \p \tau A \rb.
\end{equation}

Similarly, we look for an approximate solution to (\ref{Ut}) in the form:
\begin{align}\label{eq:Ueps}
U_\eps = \eps \Ue 1 + \eps^2 \Ue 2 + \eps^3 \Ue 3,
\end{align}	
where the functions $U^{(j)}$ depend on $(\xi,\eta)$ and decay to zero at infinity. Plugging \eqref{eq:Weps} and \eqref{eq:Ueps} into \eqref{Ut} and expanding each $W^{(j)}$ in Taylor series,  we get
\begin{align}\label{eq:UtRHS}
\begin{split}
& \eps^2\lb -c_1 \p \xi \Ue 1\rb+\eps^3 \lb -c_1 \p \xi \Ue 2\rb + \eps^4 \lb -c_1 \p \xi \Ue 3+\p \tau \Ue 1\rb\\
& \qquad = \eps^2 \p \eta \We 0 + \eps^3 \p \eta \We 1 +\eps^4 \lb \p \eta \We 2 +\frac 1 2 \p \eta^2 \We 0\rb + \mathcal{O}(\eps^5).
\end{split}	
\end{align}	
Grouping terms by their orders in powers of $\eps$ and using the values for $W_\eps$ found in \eqref{eq:Wesol}, we obtain a sequence of equations with their relevant solutions:
\begin{align}\label{eq:Uesol}
\begin{split}
& \mathcal{O}(\eps^2) : \quad -c_1 \p \xi \Ue 1 = - c_1 \p\eta A \\
&\phantom{O(\eps^2):-c_1 \p \eta U=\p\eta \We 0} \quad \implies  \Ue 1 = \p \xi^{-1}\p \eta A \\
& \mathcal{O}(\eps^3) : \quad -c_1 \p \xi \Ue 2 = \frac{c_1}{2} \p \xi\p \eta A \\
&\phantom{O(\eps^3): \frac{c_1}{2} \p \xi\p \eta U=\p\eta \We 1} \quad \implies  \Ue 2 = -\frac 1 2 \p \eta A\\
& \mathcal{O}(\eps^4) : \quad -c_1 \p \xi \Ue 3 +\p \tau \Ue 1 = 
\p \xi^{-1} \p\eta \p \tau A - \frac{c_1}{12} \p \xi^2 \p \eta A - \frac{c_1}{2} \p \eta^2 A \\
&\phantom{O(\eps^4) : \p\eta \We 2 + \frac 1 2\p \eta^2 \We 0} \quad \implies\Ue 3 = \frac 1 2 \p \xi^{-1}\p\eta^2 A +\frac{1}{12}\p \xi \p \eta A
\end{split}
\end{align}
Substituting \eqref{eq:Uesol} into \eqref{eq:Ueps} yields 
\begin{equation}\label{Ue}
U_\eps= 	\eps \p \xi^{-1}\p \eta A-\eps^2 \lb \frac 1 2 \p \eta A \rb 
+ \eps^3 \lb \frac 1 2 \p \xi^{-1}\p\eta^2 A+\frac{1}{12} \p \xi \p \eta  A\rb.
\end{equation} 

Substituting decomposition \eqref{eq:decom} into the first and second equations of system \eqref{eom} yield equations 
\begin{align}\label{U1}
\begin{split}
\dot U_{j,k}^{(1)} &= W_{j+1,k}-W_{j,k} + Res_{j,k}^{U^{(1)}},\\
\dot U^{(2)}_{j,k} &= W_{j,k+1}-W_{j,k}+Res^{U^{(2)}}_{j,k},
\end{split}
\end{align}
where
\begin{align*}
\begin{split}
Res_{j,k}^{U^{(1)}} &:=   c_1 \eps \p \xi A - \eps^3 \p \tau A + \xe {W_\eps} + - W_{\eps}\lb \xi, \eta \rb,\\
Res_{j,k}^{U^{(2)}} &:=  c_1 \eps \p\xi {U_\eps} - \eps^3 \p \tau {U_\eps}+\ye {W_\eps} + - W_{\eps}\lb \xi , \eta \rb.
\end{split}
\end{align*}
If expansions (\ref{We}) and (\ref{Ue}) are used, the residual terms have the formal order of $\mathcal{O}(\eps^5)$. The third and fourth equations of system (\ref{eom}) have not been changed:
\begin{align}
\label{U2}
\begin{split}
\dot V_{j,k}^{(1)} = Z_{j+1,k}-Z_{j,k}, \\
\dot V^{(2)}_{j,k}=Z_{j,k+1}-Z_{j,k}.
\end{split}
\end{align}

Finally, the last two equations of system (\ref{eom}) can be rewritten explicitly as
\begin{align}
\label{U3}
\begin{split}
\dot W_{j,k}&= c_1^2\lsb U^{(1)}_{j,k}-U^{(1)}_{j-1,k}\rsb+c_2^2 \lsb U^{(2)}_{j,k}-U^{(2)}_{j,k-1} \rsb\\
&\quad+  \alpha_1 \eps^2 \lsb 2A U^{(1)}_{j,k}-2\xe{A} - U^{(1)}_{j-1,k}+\lb U_{j,k}^{(1)}\rb ^2-\lb U_{j-1,k}^{(1)}\rb ^2\rsb \\
&\quad+ \alpha_2 \eps^2 \lsb U_\eps\lb \xi,\eta\rb V^{(2)}_{j,k} - \ye {U_\eps} - V^{(2)}_{j,k-1} + U^{(2)}_{j,k} V^{(2)}_{j,k} - U^{(2)}_{j,k-1}  V^{(2)}_{j,k-1} \rsb\\
&\quad+\alpha_2 \eps^2\lsb\frac 1 2 \lb V^{(1)}_{j,k}\rb^2 - \frac 1 2 \lb V^{(1)}_{j-1,k}\rb^2\rsb + Res^W_{j,k}\\
\dot Z_{j,k}&=c_2^2\lsb V_{j,k}^{(1)}-V_{j-1,k}^{(1)}\rsb + c_1^2 \lsb V_{j,k}^{(2)}-V_{j,k-1}^{(2)}\rsb\\%1
&\quad+\alpha_2 \eps^2 \lsb U_\eps\lb \xi, \eta\rb U_{j,k}^{(2)}-U_\eps\lb \xi, \eta-\eps^2\rb U_{j,k-1}^{(2)} + \frac 1 2\lb U_{j,k}^{(2)}\rb ^2-\frac 1 2\lb U_{j,k-1}^{(2)}\rb ^2 \rsb\\
&\quad+\alpha_2 \eps^2 \lsb A \lb \xi,\eta\rb V_{j,k}^{(1)}-\xe {A} - V_{j-1,k}^{(1)} + V^{(1)}_{j,k} U^{(1)}_{j,k}-V^{(1)}_{j-1,k} U^{(1)}_{j-1,k} \rsb\\
&\quad+\alpha_1 \eps^2\lsb\lb V^{(2)}_{j,k}\rb^2-\lb V^{(2)}_{j,k-1}\rb^2 \rsb + Res^Z_{j,k}
\end{split}
\end{align}
where
\begin{align*}
\begin{split}
Res^W_{j,k} := & c_1\eps \p \xi W_\eps - \eps^3 \p \tau W_\eps+ c_1^2 \lsb A \lb \xi, \eta\rb-\xe {A} -\rsb\\
& + c_2^2\lsb U_\eps\lb \xi,\eta\rb-\ye {U_\eps} - \rsb+\alpha_1\eps^2 \lsb A \lb \xi,\eta\rb^2 -  \xe {A} -^2\rsb, \\
Res^Z_{j,k} := & \frac{\alpha_2 \eps^2}{2}\lsb U_\eps \lb \xi, \eta\rb^2- \ye {U_\eps} -^2\rsb.
\end{split}
\end{align*}	
Expanding each term in $Res^W_{j,k}$ by using expansions (\ref{We}) and (\ref{Ue}) yields the following formal expansion
\begin{align*}
Res^W_{j,k}=\quad &\eps^3 \left[2c_1 \p \tau A + \frac{c_1^2}{12} \p \xi^3 A+ c_2^2 \p \xi^{-1}\p \eta^2 A+\alpha_1\p \xi \lb A ^2\rb \right]\\
-&\eps^4\left[c_1 \p \xi \p \tau A + \frac{c_1^2}{24} \p \xi^4 A +\frac {c_2^2}2 \p \eta^2 A + \frac{\alpha_1}{2} \p \xi^2 \lb A^2 \rb \right] + 
\mathcal{O}(\eps^5).
\end{align*}
If the function $A$ is a solution of the KP-II equation (\ref{KPII}), the residual term $Res^W$ has the formal order of $\mathcal{O}(\eps^5)$. It is also clear from expansion (\ref{Ue}) that $Res^Z$ has the formal order of $\mathcal{O}(\eps^6)$. 

\subsection{Step 2. Residual terms} 

The residual terms are handled by using Taylor's theorem. 
If $A$ is defined in Sobolev space $H^s(\mathbb{R}^2)$ with sufficiently large $s$, then we can estimate the residual terms in the $\ell^2(\mathbb{Z}^2)$ norm by an application of Lemma \ref{lmell2}. Since all residual terms have the formal order of $\mathcal{O}(\eps^5)$, we can obtain the 
bound of $\mathcal{O}(\eps^{\frac{7}{2}})$ on the residual terms in the $\ell^2(\mathbb{Z})$ norm.

The following lemma gives us estimates of the $\ell^2$-norm for the residual terms in equations (\ref{U1}) and (\ref{U3}). No residual terms appear in equations (\ref{U2}).

\begin{lemma}\label{lres}
Let $A \in C^0(\mathbb{R},H^s)$ be a solution to the KP-II equation \eqref{KPII}  with $s \geq 9$. There is a positive constant $C$ 
that depend on $A$ such that for all $\eps\in \lb 0, 1\rsb$, we have
	\begin{align}
	\label{bound-on-residual}
	\norm{Res^{U^{(1)}}_{j,k}}_{\ell^2}+\norm{Res^{U^{(2)}}_{j,k}}_{\ell^2} + \norm{Res^W_{j,k}}_{\ell^2}+\norm{Res^Z_{j,k}}_{\ell^2} &\leq C \eps^{\frac 7 2}.
	\end{align}
\end{lemma}

\begin{proof}
	By construction, all terms in $Res_{j,k}^{U^{(1)}}$  below the formal order of $\mathcal{O}(\eps^5)$ vanish. From the Taylor's theorem for $W_{\eps} \lb\xi+\eps,\eta \rb - W_{\eps} \lb \xi,\eta \rb$, the nonzero terms 
	at $\mathcal{O}(\eps^5)$ are given by the integrals:
	\begin{align*}
	\eps^5 \int_0^1 \p \xi^{5-l}\We l\lb \eps\lb j+r\rb,\eps^2 k\rb\frac{\lb 1-r\rb^{5-l-1}}{\lb 5-l-1\rb!}dr, \quad 0 \leq l \leq 3.
	\end{align*}
In view of corrections for $W^{(l)}$ in \eqref{We}, the error is given by a linear combination of the following two terms:
	\begin{equation*}
	\eps^5 \sup_{r \in [0,1]} \left|\p \xi^5 A\lb \eps \lb j+r\rb, \eps^2 k \rb \right|, \quad 
	\eps^5 \sup_{r \in [0,1]} \left| \p \xi^2\p \tau A\lb \eps \lb j+r\rb, \eps^2 k \rb \right|.
	\end{equation*}
Using Lemma \ref{lmell2}, there is a constant $C_s > 0$ such that the $\ell^2$ norm of the residual term $Res^{U^{(1)}}$ is bounded by 
	$$
	\norm{Res^{U^{(1)}}}_{\ell^2} \leq C_s \eps^{\frac{7}{2}} \lb \| A \|_{H^{s+5}} + \| \p \tau A\|_{H^{s+2}} \rb,
	$$
	for $s > 1$.
	
	Similarly, all terms in $Res_{j,k}^{U^{(2)}}$ below the formal order of $\mathcal{O}(\eps^5)$ vanish. From the Taylor's theorem for $W_{\eps} \lb\xi,\eta + \eps^2 \rb - W_{\eps} \lb \xi,\eta \rb$ and the corrections for $W_{\eps}$ and $U_{\eps}$ given by (\ref{We}) and \eqref{Ue}, the error is given by a linear combination of the following six terms:
	\begin{align*}
	& \eps^5 \sup_{r \in [0,1]} \left|\p \xi \p \eta^2  A \lb \eps j, \eps^2(k+r) \rb \right|, \quad &	\eps^5 \sup_{r \in [0,1]} \left| \p \eta \p \tau A \lb \eps j, \eps^2(k+r) \rb \right|, \\
	&  \eps^6 \sup_{r \in [0,1]} \left| \p \eta^3 A \lb \eps j, \eps^2(k+r) \rb \right|, \quad & \eps^6 \sup_{r \in [0,1]} \left|\p \xi^2  \p \eta^2 A \lb \eps j, \eps^2(k+r) \rb \right|, \\
		&  \eps^6 \sup_{r \in [0,1]} \left| \p \xi^{-1} \p \eta^2 \p \tau A \lb \eps j, \eps^2(k+r) \rb \right|, \quad & \eps^6 \left| \p \xi \p \eta  \p \tau A \lb \eps j, \eps^2 k \rb \right|.
	\end{align*}
Using Lemma \ref{lmell2}, there is a constant $C_s > 0$ such that the $\ell^2$ norm of the residual term $Res^{U^{(2)}}$ is bounded for $\eps \in (0,1]$ by 
	$$
	\norm{Res^{U^{(2)}}}_{\ell^2} \leq C_s \eps^{\frac{7}{2}} \lb \| A \|_{H^{s+4}} + \| \p \tau A\|_{H^{s+2}}  + \| \p \xi^{-1} \p \eta \p \tau A \|_{H^{s+1}} \rb,
	$$
for $s > 1$.
	
If $A \in C^0(\mathbb{R},H^s)$ is a solution to the KPII equation (\ref{KPII}) 
with $s \geq 9$, all terms in $Res_{j,k}^{W}$ below the formal order of $\mathcal{O}(\eps^5)$ vanish. Expanding all terms in $Res_{j,k}^{W}$ 
shows that the error is given by a linear combination of the following eight terms:
\begin{align*}
\eps^5 \left|\p \xi^{-1} \p \tau^2 A \lb \eps j, \eps^2 k \rb \right|, \quad 	\eps^5 \left|\p \xi^2 \p \tau A \lb \eps j, \eps^2 k \rb \right|, \quad 
\eps^6 \left| \p \tau^2 A \lb \eps j, \eps^2 k \rb \right|,
\end{align*}
	\begin{align*}
\eps^5 \sup_{r \in [0,1]} \left|\p \xi^5 A \lb \eps (j+r), \eps^2 k \rb \right|, \quad \eps^5 \sup_{r_1,r_2 \in [0,1]} \left| A\lb \eps (j+r_1),\eps^2 k \rb \p \xi^3 A \lb \eps (j+r_2), \eps^2 k \rb \right|,
	\end{align*}
	\begin{align*}
\eps^5 \sup_{r \in [0,1]} \left| 
\p \xi^{-1} \p \eta^3 A \lb \eps j, \eps^2(k+r) \rb \right|,  \quad 
\eps^5 \sup_{r \in [0,1]} \left|\p \xi \p \eta^2  A \lb \eps j, \eps^2 (k+r) \rb \right|, 
\end{align*}
and
\begin{align*}
\eps^6 \sup_{r \in [0,1]} \left| 
\p \eta^3 A \lb \eps j, \eps^2(k+r) \rb \right|.
\end{align*}
Using Lemma \ref{lmell2} and the relation (\ref{A-der}) for $\p \xi^{-1} \p \tau^2 A$, there is a constant $C_s > 0$ such that the $\ell^2$ norm of the residual term $Res^{W}$ is bounded for $\eps \in (0,1]$ by
\begin{align*}
\norm{Res^{W}}_{\ell^2} &\leq C_s \eps^{\frac 7 2} \left( \| \p \tau^2 A \|_{H^{s}} + \| \p \tau A\|_{H^{s+2}} + \| \p \xi^{-2} \p \eta^2 \p \tau  A\|_{H^{s}} + \|  A \|_{H^{s+5}} + \| \p \xi^{-1} \p \eta A \|_{H^{s+2}} \right. \\
& \qquad \qquad \left. +\| A \|_{H^s}  \| A\|_{H^{s+3}} + \| A \|_{H^s} \| \p \tau A \|_{H^s} \right),
\end{align*}
for $s > 1$.

The final residual term $Res^Z_{j,k}$ is estimated from 
the expansion (\ref{Ue}). The error is given by
\begin{align*}
& \eps^6 \sup_{r_1,r_2,r_3,r_4 \in [0,1]} 
\left( | \p \xi^{-1} \p \eta A(\eps j,\eps^2 (k+r_1)) | 
+ \eps |\p \eta A(\eps (j+r_2),\eps^2 k)| \right) \\
& \qquad \qquad 
\left( | \p \xi^{-1} \p \eta^2 A(\eps j,\eps^2 (k+r_3)) | 
+ \eps |\p \eta^2 A(\eps (j+r_4),\eps^2 k)| \right) 
\end{align*}
and is controlled in the $\ell^2$ norm by using the bound	
$$
	\norm{Res^{Z}}_{\ell^2} \leq C_s	\eps^{\frac 9 2} \left( 
	\norm{\p \xi^{-1} \p \eta  A}^2_{H^{s+1}} + 
		\norm{A}^2_{H^{s+2}} \right),
	$$
where $C_s > 0$, $s > 1$, and $\eps \in (0,1]$.
	
Combining all four bounds together and using the fact that 
a solution $A \in C^0(\mathbb{R},H^s)$ to the KP-II equation \eqref{KPII} with $s \geq 9$	enjoys the estimates of Lemma \ref{welpos}, we obtain 
the bound (\ref{bound-on-residual}).
\end{proof}

\subsection{Step 3. Energy estimates}

In order to control the growth of the approximation error from solutions to system (\ref{U1}), (\ref{U2}), and (\ref{U3}), we will introduce the following energy function,
	\begin{align}\label{alpha-energy}
	\begin{split}
	E(t)=&\frac 1 2 \sum_{j,k\in \Z^2} W_{j,k}^2 + Z_{j,k}^2 +c_1^2 \lb U_{j,k}^{(1)}\rb^2+c_2^2 \lb U_{j,k}^{(2)}\rb^2 +c_1^2 \lb V_{j,k}^{(1)}\rb^2+c_2^2 \lb V_{j,k}^{(2)}\rb^2 \\
	&+\alpha_1 \eps^2 \lsb  2 A \lb U_{j,k}^{(1)}\rb^2+\frac 2 3\lb U_{j,k}^{(1)}\rb^3+\frac 2 3\lb V_{j,k}^{(2)}\rb^3\rsb\\
	&+\alpha_2 \eps^2 \lsb A \lb V_{j,k}^{(1)}\rb^2 +U_{j,k}^{(1)} \lb V_{j,k}^{(1)}\rb^2+ \lb U_{j,k}^{(2)} \rb^2 V_{j,k}^{(2)} + 2 U_\eps U_{j,k}^{(2)} V_{j,k}^{(2)} \rsb.
	\end{split}
	\end{align}
The $\eps$-dependent terms of $E(t)$ are chosen from the condition that the growth rate $E'(t)$ along the solution of system (\ref{U1}), (\ref{U2}), and (\ref{U3}) does not contain terms of the formal orders $\mathcal{O}(\eps)$ and $\mathcal{O}(\eps^2)$ (see Lemma \ref{lemma-growth} below). 

The following lemma establishes coercivity of the energy $E(t)$ with respect to the $\ell^2$ norm of the perturbations as long as the perturbations are not large in the $\ell^2$ norm.

\begin{lemma}
	\label{Eest}
	Let $A \in C^0\lb \lsb -\tau_0, \tau_0\rsb,H^{s}\lb\R^2\rb\rb$ 
	and $\p \xi^{-1} \p \eta A \in C^0\lb \lsb -\tau_0, \tau_0\rsb,H^{s-1}\lb\R^2\rb\rb$ with $s > 3$ and assume that $E(t) \leq E_0$ for some $\eps$-independent constant $E_0 > 0$ for every $t \in [-\tau_0 \eps^{-3}, \tau_0 \eps^{-3}]$. 
	There exists some constants $\eps_0>0$ and $K_0 >0$ that depend on $A$ such that 
\begin{equation}
\label{energy-coercive}
	\elt W^2+\elt Z^2+ \elt{\UV U 1}^2+ \elt{\UV U 2}^2+ \elt{\UV V 1}^2+ \elt{\UV V 2}^2\leq 2 K_0 E(t),
\end{equation}
	for each $\eps\in\lb 0,\eps_0\rb$ and $t\in [-\tau_0 \eps^{-3}, \tau_0 \eps^{-3}]$.
\end{lemma}

\begin{proof}
It follows from the decomposition (\ref{Ue}) and Sobolev's embedding 
of $H^s(\R^2)$ into $L^{\infty}(\R^2)$ for $s > 1$ that if $A \in C^0\lb \lsb -\tau_0, \tau_0\rsb,H^{s}\lb\R^2\rb\rb$ 
and $\p \xi^{-1} \p \eta A \in C^0\lb \lsb -\tau_0, \tau_0\rsb,H^{s-1}\lb\R^2\rb\rb$ with $s > 3$, then there exists 
a constant $C_0 > 0$ that depends on $A$ such that 
	\begin{align*}
\sup_{t \in [-\tau_0,\tau_0]} \left( \|A(\cdot,\tau) \|_{L^{\infty}(\R^2)} + 
\|U_{\eps}(\cdot,\tau) \|_{L^{\infty}(\R^2)} \right) \leq C_0.
	\end{align*}
By using (\ref{alpha-energy}) and the Cauchy--Schwarz inequality for $U^{(2)}_{j,k} V^{(2)}_{j,k}$, we derive the following lower bound for $E(t)$:
	\begin{align*}
	2E(t)\geq & \elt {W}^2 + \elt Z^2 \\
	& + \lb c_1^2 - 2 \alpha_1 \eps^2 C_0 -\frac 2 3 \alpha_1 \eps^2 \| U^{(1)} \|_{\ell^{\infty}} \rb \elt {\UV U 1}^2\\
	&+\lb c_2^2 -\alpha_2 \eps^2 C_0- \alpha_2 \eps^2 \| V^{(2)} \|_{\ell^{\infty}} \rb \elt {\UV U 2}^2\\
	&+\lb c_1^2 -\alpha_2 \eps^2 C_0 -\alpha_2 \eps^2 \| U^{(1)} \|_{\ell^{\infty}}  \rb \elt {\UV V 1}^2\\
	&+\lb c_2^2 -\alpha_2 \eps^2 C_0 - \frac 2 3 \alpha_1 \eps^2 \| V^{(2)} \|_{\ell^{\infty}}  \rb \elt {\UV V 2}^2
	\end{align*}
	For fixed $c_1,c_2>0$, we use the bound $\| U \|_{\ell^{\infty}} \leq \| U \|_{\ell^2}$ 
	and choose $\eps_0 > 0$ and $K_0 > 0$ such that 
	\begin{align*}
	& c_1^2 - 2 \alpha_1 \eps_0^2 C_0 -\frac 2 3 \alpha_1 \eps_0^2 (2K_0 E_0)^{1/2} \geq {\rm min}(1,K_0^{-1}), \\
	& c_2^2 -\alpha_2 \eps_0^2 C_0- \alpha_2 \eps_0^2 (2K_0 E_0)^{1/2} \geq {\rm min}(1,K_0^{-1}), \\
	& c_1^2 -\alpha_2 \eps_0^2 C_0 -\alpha_2 \eps_0^2 (2K_0 E_0)^{1/2} \geq {\rm min}(1,K_0^{-1}), \\
	& c_2^2 -\alpha_2 \eps_0^2 C_0 - \frac 2 3 \alpha_1 \eps_0^2 (2K_0 E_0)^{1/2} \geq {\rm min}(1,K_0^{-1}), 
	\end{align*}
	which is always possible if $\eps_0$ is sufficiently small. This gives the bound (\ref{energy-coercive}) with $K_0$ redefined as $\max(1,K_0)$.
\end{proof}

The following lemma uses the coercivity of the energy $E(t)$ in Lemma \ref{Eest} to establish the rate at which it may grow in time along the solutions of system (\ref{U1}), (\ref{U2}), and (\ref{U3}). We will be able to use this in Step 4, along with a Gronwall lemma argument of Lemma \ref{gronwall} below, in order to get a bound on the size of the energy quantity. This will in turn gives a bound on how far solutions of the KP-II equation \eqref{KPII} drift away from solutions of the FPU system \eqref{eom} and hence will complete the proof of Theorem \ref{theorem-horiz}.

\begin{lemma}
\label{lemma-growth}
	Let $A \in C^0\lb \lsb -\tau_0, \tau_0\rsb,H^{s+9}\lb\R^2\rb\rb$ be a solution to the KP-II equation \eqref{KPII} with $s \geq 0$ in the class of functions of Lemma \ref{KPIIwellpos} and assume that $E(t) \leq E_0$ for some $\eps$-independent constant $E_0 > 0$ for every $t \in [-\tau_0 \eps^{-3},\tau_0 \eps^{-3}]$. There exist some constants $\eps_0 > 0$ and $K_0 > 0$ that depend on $A$ such that 
	\begin{equation}
	\label{energy-balance}
	\left|E'(t)\right|\leq K_0 \lb \eps^{\frac 7 2} E(t)^{\frac 1 2}  + \eps^3 E(t) \rb,
	\end{equation}	
	for each $\eps \in (0,\eps_0)$ and $t \in [-\tau_0 \eps^{-3},\tau_0 \eps^{-3}]$.
\end{lemma}

\begin{proof}	By differentiating $E(t)$, defined by \eqref{alpha-energy}, in time $t$, we obtain
	\begin{align*}
E'(t) = &\sum_{j,k\in \Z^2} \dot W_{j,k}W_{j,k} + \dot Z_{j,k}Z_{j,k}+c_1^2 \dot U_{j,k}^{(1)} U_{j,k}^{(1)}+c_2^2 \dot U_{j,k}^{(2)} U_{j,k}^{(2)} + c_1^2 \dot V_{j,k}^{(1)} V_{j,k}^{(1)}+c_2^2 \dot V_{j,k}^{(2)}V_{j,k}^{(2)}\\%1
	& +\alpha_1 \eps^2 \lsb \dot {A} \lb U_{j,k}^{(1)}\rb^2+ 2 A \dot U_{j,k}^{(1)}U_{j,k}^{(1)}+\dot U_{j,k}^{(1)}\lb U_{j,k}^{(1)}\rb^2+\dot  V_{j,k}^{(2)}\lb V_{j,k}^{(2)}\rb^2\rsb\\%2
	&+\alpha_2 \eps^2 \lsb  \dot{U_\eps}U_{j,k}^{(2)}V_{j,k}^{(2)}+\frac 1 2 \dot U_{j,k}^{(1)} \lb V_{j,k}^{(1)}\rb^2+U_{j,k}^{(1)} \dot V_{j,k}^{(1)}V_{j,k}^{(1)}+\frac 1 2 \lb U_{j,k}^{(2)} \rb^2 \dot V_{j,k}^{(2)}\right.\\%4
	&\left.+\dot U_{j,k}^{(2)}U_{j,k}^{(2)} V_{j,k}^{(2)} +U_\eps\dot U_{j,k}^{(2)}V_{j,k}^{(2)}+U_\eps U_{j,k}^{(2)}\dot V_{j,k}^{(2)}+\frac 1 2 \dot {A}\lb V_{j,k}^{(1)}\rb^2+ A \dot V_{j,k}^{(1)}V_{j,k}^{(1)}\rsb,
	\end{align*}
	where the dot denotes derivative in $t$ and is applied with the chain rule to $A$ and $U_{\eps}$ that depends on $(\eps(j-c_1t),\eps^2 k,\eps^3 t)$.
Substituting equations of motion \eqref{U1}, \eqref{U2}, and \eqref{U3} and summing across $(j,k)\in \Z^2$ simplifies $E'(t)$ to the form:
	\begin{align*}
E'(t) = &\sum_{j,k\in \Z^2} W_{j,k} Res^{W}_{j,k} +Z_{j,k} Res^{Z}_{j,k} \\
& + \lb c_1^2 \Ux j k + 2\alpha_1 \eps^2 A \Ux j k + \alpha_1\eps^2 \lb U_{j,k}^{(1)} \rb^2+ \frac{1}{2} \alpha_2\eps^2  \lb V_{j,k}^{(1)}\rb^2 \rb
Res^{U^{(1)}} \\
& + \lb c_2^2 \Uy j k +\alpha_2 \eps^2 U_\eps \Vy j k + \alpha_2\eps^2  U_{j,k}^{(2)} V_{j,k}^{(2)} \rb Res^{U^{(2)}} \\
& +\alpha_1 \eps^2 \lb -c_1\eps \p \xi A + \eps^3 \p \tau A \rb\lb \Ux j k\rb^2\\
& +\alpha_2 \eps^2 \lb  -c_1 \eps\p\xi  {U_\eps} + \eps^3 \p \tau  {U_\eps}\rb \Uy j k \Vy j k\\
&+\frac{\alpha_2}2 \eps^2 \lb - c_1\eps \p \xi A + \eps^3 \p \tau A  \rb\lb \Vx j k\rb^2.
	\end{align*}
Applying the Cauchy-Schwartz inequality  and the bound $\| U \|_{\ell^{\infty}} \leq \| U \|_{\ell^2}$ yields
\begin{align*}
\left| E'(t) \right|\leq& \elt{W} \elt{Res^W}+\elt Z  \elt{Res^Z}\\%1
&+ \lb c_1^2 \elt {\UV U 1} 
+ 2 \alpha_1 \eps^2 \Li {A} \elt{\UV U 1} 
+ \alpha_1 \eps^2 \elt{U^{(1)}}^2
+ \frac{1}{2} \alpha_2 \eps^2 \elt{V^{(1)}}^2
\rb \elt{Res^{\UV U 1}} \\ 
& + \lb c_2^2 \elt{\UV U 2} + \alpha_2 \eps^2 \Li{U_\eps} \elt{\UV V 2} 
+ \alpha_2 \eps^2 \| U^{(2)} \|_{\ell^2} \| V^{(2)} \|_{\ell^{2}}
\rb  \elt{Res^{\UV U 2}}\\
&+ \eps^2 \lb c_1 \eps \Li{\p \xi A} + \eps^3 \Li{\p \tau A} \rb 
\lb \alpha_1 \elt{\UV U 1}^2 + \frac{1}{2} \alpha_2 \elt{\UV V 1}^2 \rb \\
&+ \eps^2 \lb c_1 \eps \Li{\p \xi U_{\eps}} + \eps^3 \Li{\p \tau U_{\eps}} \rb 
\elt {\UV U 2} \elt{\UV V 2}.
	\end{align*}
Estimating the residual terms and the perturbation terms with the help of Lemmas \ref{lres} and \ref{Eest}, respectively, yields (\ref{energy-balance}).
\end{proof}

\subsection{Step 4. Bound on the approximation error}

By making the substitution $E(t) := \frac{1}{2} Q(t)^2$, we obtain 
from (\ref{energy-balance}):
\begin{equation}
\label{norm-balance}
\left| Q'(t) \right| \leq K_0 \lb \eps^{\frac 7 2}+ \eps^3Q\rb,
\end{equation}
where the constant $K_0$ may change from one line to another line. 
The norm of the perturbation terms controlled by $Q(t)$ is further estimated by using the Gronwall's inequality.

\begin{lemma}\label{gronwall}
Assume that $Q(t)$ satisfies (\ref{norm-balance}) for $t \in [-\tau_0 \eps^{-3},\tau_0 \eps^{-3}]$ and $Q(0)\leq C_0 \eps^{\frac 1 2}$ for some $\eps$-independent constant $C_0$. There exists $\eps_0 > 0$ such that 
	\begin{equation}
	\label{Q-control}
	Q(t)\leq \eps^{\frac{1}{2}} (1 + C_0) \exp\lb K_0 \tau_0\rb
	\end{equation}
	for each $\eps \in \lb 0,\eps_0\rb$ and $t \in [-\tau_0 \eps^{-3},\tau_0 \eps^{-3}]$.
\end{lemma}

\begin{proof}
By using the integrating factor, we can rewrite (\ref{norm-balance}) in the form:
	\begin{equation}\label{Qdiff1}
	\frac{d}{dt}\lsb \exp\lb - \eps^3 K_0 t\rb Q\rsb\leq K_0 \eps^{\frac 7 2}\exp\lb - \eps^3 K_0 t\rb.
	\end{equation}
	Integrating \eqref{Qdiff1} yields the Gronwall's inequality
	\begin{equation}\label{Qdiff2}
	Q(t)\leq \lb Q(0)+ \eps^{\frac 1 2}\rb \exp\lb  \eps^3 K_0 t\rb.
	\end{equation}
	Since $Q(0)\leq C\eps^{\frac 1 2}$ and $t \in [-\tau_0 \eps^{-3},\tau_0 \eps^{-3}]$, the inequality (\ref{Qdiff2}) yields (\ref{Q-control}).
\end{proof}

We can finish the proof of Theorem \ref{theorem-horiz} by using Lemmas \ref{Eest} and \ref{gronwall}. Since 
\begin{align*}
Q(0)\leq C_{A} \lb\norm{U^{(1)}_{in}}_{\ell^2}+\norm{U^{(2)}_{in}}_{\ell^2}+\norm{W_{in}}_{\ell^2} +\norm{V^{(1)}_{in}}_{\ell^2} +\norm{V^{(2)}_{in}}_{\ell^2}+\norm{Z_{in}}_{\ell^2}\rb,
	\end{align*}
	where the subscript $in$ stands for the initial condition, 
the decomposition \eqref{eq:decom} and the hypothesis \eqref{hyp} yield that 
$Q(0)\leq C \eps^{\frac 1 2}$, so that Lemma \ref{gronwall} gives the bound (\ref{Q-control}). With decomposition \eqref{eq:decom} and Lemma \ref{Eest} we have that 
	\begin{align*}
	&\norm{u^{(1)}(t)-\eps^2 A\lb \eps(j-c_1t),\eta^2k,\eps^3 t\rb}^2_{\ell^2}+\norm{u^{(2)}(t)-\eps^2 U_\eps\lb \eps(j-c_1t),\eta^2k,\eps^3 t\rb}^2_{\ell^2}\\
	& + \norm{w(t)-\eps^2 W_\eps\lb \eps(j-c_1t),\eta^2k,\eps^3 t\rb}^2_{\ell^2} +\norm{v^{(1)}(t)}^2_{\ell^2}+\norm{v^{(2)}(t)}^2_{\ell^2}+\norm{z(t)}^2_{\ell^2}\leq K_0 \eps^4 Q(t)^2,
	\end{align*}
	where the $\eps$-dependent functions $W_{\eps}$ and $U_{\eps}$ are given in terms of $A$ by (\ref{We}) and (\ref{Ue}), respectively.
Due to the bound (\ref{Q-control}) and the triangle inequality, we obtain 
the bound (\ref{res}) and the result of Theorem \ref{theorem-horiz} follows.

\section{Proof of Theorem \ref{theorem-diag}}
\label{sec-proof-2}

Here we give relevant details of the proof of Theorem \ref{theorem-diag}. As is explained in 
the introduction, we will only consider the reduction of the general FPU system 
if $c_1 = c_2$ and $\alpha_2 = 2 \alpha_1$, when the vertical and horizontal displacements 
on the square two-dimensional lattice coincide with $x_{j,k} = y_{j,k}$. 

Instead of working with the strain variables in (\ref{strain-variabl-diag}), 
we introduce the following strain variables (see Fig. \ref{fig:2D-diag}):
\begin{align}
\left\{ \begin{array}{l} 
a^l_{m,n} =\chi_{m,n}-x_{m,n},\\
a^d_{m,n} =x_{m+1,n+1}-\chi_{m,n},\\
a^x_{m,n} =x_{m+1,n}-\chi_{m,n},\\
a^y_{m,n} =x_{m,n+1}-\chi_{m,n},
\end{array} \right.
\end{align}
in order to write equations of motion in the form:
\begin{align}
\label{diagaccx}
\begin{split}
\dot{a}^l_{m,n} =& v_{m,n} - u_{m,n},\\
\dot{a}^d_{m,n} =& u_{m+1,n+1} - v_{m,n},\\
\dot{a}^x_{m,n} =& u_{m+1,n} - v_{m,n},\\
\dot{a}^y_{m,n} =& u_{m,n+1} - v_{m,n},\\
\dot u_{m,n} =& c_1^2 \lb \av l m {n}- \av d {m-1} {n-1} - \av x {m-1} n - \av y {m} {n-1}\rb\\&
+2\alpha_1 \lsb \lb \av l m {n}\rb^2-\lb \av d {m-1} {n-1}\rb^2 -\lb\av x{m-1}{n}\rb^2 + \lb\av y m {n-1} \rb^2 \rsb, \\
\dot v_{m,n} =& c_1^2 \lb \av d m {n}- \av l {m} n + \av x m n + \av y {m} {n}\rb\\
&+2\alpha_1 \lsb \lb \av d m {n}\rb^2-\lb \av l {m} n\rb^2+\lb\av x m n \rb^2-\lb\av y {m} {n}\rb^2\rsb,
\end{split}
\end{align}
where $u_{m,n} := \dot{x}_{m,n}$, $v_{m,n} := \dot{\chi}_{m,n}$, and $(m,n) \in \mathbb{Z}^2$. The justification procedure is divided into the same four steps as in the case of the horizontal propagation. 

\subsection{Step 1. Decomposition} 
We will use the following decomposition,
\begin{align}\label{diagdecom}
\begin{split}
a^l_{m,n}&=\eps^2 L_\eps\lb \xi,\eta,\tau\rb+\eps^2 L_{m,n},\\
a^d_{m,n}&=\eps^2 D_\eps\lb \xi,\eta,\tau\rb+\eps^2 D_{m,n},\\
a^x_{m,n}&=\eps^2 X_\eps\lb \xi,\eta,\tau\rb+\eps^2 X_{m,n},\\
a^y_{m,n}&=\eps^2 Y_\eps\lb \xi,\eta,\tau\rb+\eps^2 Y_{m,n},\\
u_{m,n}&=\eps^2 U_\eps\lb \xi,\eta,\tau\rb+\eps^2 U_{m,n},\\
v_{m,n}&=\eps^2 V_\eps\lb \xi,\eta,\tau\rb+\eps^2 V_{m,n},
\end{split}
\end{align}
where $\xi = \eps (m - c_1^* t)$, $\eta = \eps^2 n$, and $\tau= \eps^3 t$ with $c_1^* := \frac{c_1}{\sqrt{2}}$. By ignoring the error terms and the residual terms of the formal order of $\mathcal{O}(\eps^5)$ for the time being, we shall use the equations of motion  \eqref{diagaccx} and define the expansions of the functions $L_{\eps}, \dots, V_{\eps}$ in $\eps$ from the condition that all residual terms of the formal order below $\mathcal{O}(\eps^5)$ are removed.

The first equation in system \eqref{diagaccx} gives us the relationship: 
\begin{equation}
\label{Vep}
-\eps c_1^* \partial_{\xi} L_{\eps} + \eps^3 \partial_{\tau} L_{\eps} = V_{\eps} - U_{\eps},
\end{equation}
which is used to eliminate $V_{\eps}$ from all other relations. 

Adding the first and third equations in system (\ref{diagaccx}) implies
\begin{equation}
\label{Ueq}
-\eps c_1^* \partial_{\xi} (L_{\eps} + X_{\eps}) + \eps^3 \partial_{\tau} (L_{\eps} + X_{\eps}) = U_\eps \lb\xi+\eps,\eta\rb-U_\eps\lb\xi,\eta\rb,
\end{equation}
which coincides with equation (\ref{WxR}) up to notations. As follows from (\ref{KPII-scaling-diag}), we set 
\begin{equation}
\label{Xeq}
X_{\eps} + L_{\eps} = A,
\end{equation} 
where $A$ is a suitable solution of the KP-II equation (\ref{KPIIdiag}). 
Since (\ref{Ueq}) coincides with (\ref{WxR}), we rewrite 
expansions (\ref{eq:Weps}) and (\ref{We}) in new notations:
\begin{align}
\label{Uep}
U_\eps = -c_1^* A + \eps \lb \frac{c^*_1}{2}\p \xi A \rb + \eps^2 \lb\p \xi^{-1} \p \tau A-\frac{c^*_1}{12} \p\xi^2 A\rb-\eps^3 \lb \frac 1 2 \p \tau A \rb.
\end{align}

Adding the first and fourth equations in system \eqref{diagaccx} implies
\begin{equation}\label{Deq}
-\eps c_1^* \partial_{\xi} (L_{\eps} + Y_{\eps}) + \eps^3 \partial_{\tau} (L_{\eps} + Y_{\eps}) = \efun U\lb \xi, \eta+\eps^2\rb -\efun U\lb \xi,\eta\rb
\end{equation}
which coincides with equation (\ref{Ut}) up to notations, Again, we rewrite 
expansions (\ref{eq:Ueps}) and (\ref{Ue}) in new notations:
\begin{align}
\label{Yep}
Y_{\eps} + L_{\eps} = \eps \p \xi^{-1}\p \eta A-\eps^2 \lb \frac 1 2 \p \eta A \rb 
+ \eps^3 \lb \frac 1 2 \p \xi^{-1}\p\eta^2 A+\frac{1}{12}\p \eta \p \xi A\rb.
\end{align}

Finally, adding the first and second equations in system \eqref{diagaccx} implies
\begin{equation}
\label{Beq}
-\eps c_1^* \partial_{\xi} (L_{\eps} + D_{\eps}) + \eps^3 \partial_{\tau} (L_{\eps} + D_{\eps}) = \efun U\lb \xi+\eps, \eta+\eps^2\rb -\efun U\lb \xi,\eta\rb.
\end{equation}
We derive by using Taylor series and expansion \eqref{Uep} up 
to the formal order of $\mathcal{O}(\eps^5)$:
\begin{align*}
\efun U\lb \xi+\eps, \eta+\eps^2\rb -\efun U\lb \xi,\eta\rb &= \eps \p \xi \efun U + \eps^2 \lb\ \frac 1 2 \p \xi^2 \efun U +\p\eta U_\eps\rb \\
& \quad +\eps^3 \lb\frac 1 6 \p \xi^3 \efun U + \p \xi \p \eta \efun U\rb + \eps^4 \lb \frac 1 {24} \p \xi^4 \efun U + \frac 1 2 \p\xi^2 \p \eta \efun U + \frac 1 2 \p\eta^2 \efun U\rb \\
&= -c_1^*\eps  \p\xi A-c_1^*\eps^2 \p \eta A+\eps^3 \lb\p \tau A - \frac{c_1^*}{2}\p \xi \p \eta A\rb\\
& \quad +\eps^4 \lb\pxi 1 \p \eta \p\tau A - \frac{c_1^*}{12}\p\xi^2 \p \eta A - \frac {c_1^*}{2}\p\eta^2 A\rb.
\end{align*}
Expanding the left hand side of equation \eqref{Beq} in orders of $\eps$ and comparing with the previous expansions yields 
\begin{align}
\label{Dep}
D_{\eps} + L_{\eps} = A + \eps \p \xi^{-1}\p \eta A + 
\eps^2 \lb \frac 1 2 \p \eta A \rb 
+ \eps^3 \lb \frac 1 2 \p \xi^{-1}\p\eta^2 A+\frac{1}{12} \p \xi \p \eta A\rb.
\end{align}

All quantities of the decomposition (\ref{diagdecom}) are now defined in terms of $A$ and $L_{\eps}$. We can now use the fifth and sixth equations in system (\ref{diagaccx}) in order to define $L_{\eps}$ and to verify the validity of the KP-II equation (\ref{KPIIdiag}) for $A$ up to truncation at the formal order of $\mathcal{O}(\eps^5)$. The fifth and sixth equations in system \eqref{diagaccx} yield
\begin{equation}\label{Eeq}
\begin{split}
& - \eps c_1^* \p \xi \efun U +\eps^3 \p \tau \efun U = c_1^2 \lb \efun L \lb \xi, \eta\rb -\efun D \lb \xi-\eps, \eta-\eps^2\rb -\efun X \lb \xi-\eps, \eta\rb-\efun Y \lb \xi, \eta-\eps^2\rb\rb \\
&+2\alpha_1 \eps^2\lsb \efun L \lb \xi, \eta\rb^2 - \efun D \lb \xi-\eps, \eta-\eps^2\rb^2 -\efun X \lb \xi-\eps, \eta\rb^2 + \efun Y \lb \xi, \eta-\eps^2\rb^2\rsb
\end{split}
\end{equation}
and
\begin{equation}\label{Feq}
\begin{split}
& - \eps c_1^*  \p \xi \efun V+\eps^3 \p \tau \efun V = c_1^2 \lb\efun D \lb \xi, \eta\rb - \efun L \lb \xi, \eta\rb + \efun X \lb \xi, \eta\rb+\efun Y \lb \xi, \eta\rb\rb \\
&+2\alpha_1 \eps^2\lsb   \efun D \lb \xi, \eta\rb^2- \efun L \lb \xi, \eta\rb^2 + \efun X \lb \xi, \eta\rb^2- \efun Y \lb \xi, \eta\rb^2 \rsb
\end{split}
\end{equation}

In view of equation (\ref{Vep}), the left-hand side of equation (\ref{Feq}) is expanded as
\begin{align*}
- \eps c_1^*  \p \xi \efun V+\eps^3 \p \tau \efun V =-\eps c_1^* \p \xi \efun U + \eps^2 (c_1^*)^2\p\xi^2 \efun L +\eps^3 \p \tau \efun U -2\eps^4 c_1^* \p\xi\p\tau \efun L + \eps^6\p\tau^2 \efun L,
\end{align*}
whereas the right-hand side of equation (\ref{Feq}) can be written as 
\begin{align*}
& c_1^2 \lb\efun D + \efun L + \efun X + \efun L + \efun Y + \efun L - 4 \efun L \rb \\
&+2\alpha_1 \eps^2\lsb  (\efun D + \efun L)^2 + (\efun X + \efun L)^2 - (\efun Y + \efun L)^2 - 2 \efun L \lb \efun D + \efun L + \efun X + \efun L - \efun Y - \efun L \rb \rsb
\end{align*}
We expand $L_{\eps}$ in powers of $\eps$ as 
\begin{equation}\label{eq:Leps}
L_\eps= \Le 0 +\eps \Le 1 + \eps^2 \Le 2+\eps^3 \Le 3,
\end{equation}
where the functions $L^{(j)}$ depend on $(\xi,\eta)$ and decay to zero at infinity. Substituting (\ref{Xeq}), (\ref{Uep}), (\ref{Yep}), (\ref{Dep}), and (\ref{eq:Leps}) into the left-hand and right-hand sides of equation (\ref{Feq}) yields the following equations in different powers of $\eps$ with their corresponding solutions:
\begin{align*}%\label{eq:Lesol}
& \mathcal{O}(1) : \quad 0 = 2(c_1^*)^2 (2A - 4 \Le 0)\\
&\phantom{O(1): \p\xi\We 0 } \qquad \implies \Le 0 = \frac{1}{2} A\\
& \mathcal{O}(\eps) : \quad (c_1^*)^2 \p \xi A = 2(c_1^*)^2 (2 \p \xi^{-1} \p \eta A - 4 \Le 1) \\
&\phantom{O(\eps): \p\xi\We 0 } \qquad  \implies \Le 1 = \frac{1}{2} \p \xi^{-1} \p \eta A - \frac{1}{8} \p \xi A\\
& \mathcal{O}(\eps^2) : \quad 0 = 2(c_1^*)^2 (- 4 \Le 2)  \\
&\phantom{O(\eps^2): \p\xi\We 0 } \qquad  \implies \Le 2 = 0 \\
&  \mathcal{O}(\eps^3) : \quad -2c_1^*\p\tau A - \frac{(c_1^*)^2}{24} \p\xi^3 A + \frac{(c_1^*)^2}{2}\p\xi \p\eta A = 2 (c_1^*)^2 ( \p \xi^{-1} \p \eta^2 A + \frac{1}{6} \p \xi \p \eta A - 4 \Le 3) + \alpha_1 A \p\xi A. 
\end{align*}
By using the KPII equation (\ref{KPIIdiag}), we eliminate 
\begin{align*}
2c_1^* \p\tau A + \alpha_1 A \p \xi A = -\frac{1}{48} (c_1^*)^2 \p \xi^3 A - (c_1^*)^2 \p \xi^{-1} \p \eta^2 A
\end{align*}
and obtain from the equation at the order of $\mathcal{O}(\eps^3)$ that 
$$
\Le 3 =  \frac 1 8 \pxi 1 \p \eta^2 A 
+\frac 1 {384}\p\xi^3 A 
-\frac 1 {48}\p\xi \p\eta A.
$$
Substituting this expansion into (\ref{eq:Leps}) yields 
yields the expansion
\begin{equation}
\label{Lep}
L_{\eps} = \frac{1}{2} A + \eps \left(\frac{1}{2} \p \xi^{-1} \p \eta A - \frac{1}{8} \p \xi A \right) + \eps^3 \left( \frac 1 8 \pxi 1 \p \eta^2 A 
+\frac 1 {384}\p\xi^3 A 
-\frac 1 {48}\p\xi \p\eta A\right).
\end{equation}
By using (\ref{Xeq}), (\ref{Yep}), and (\ref{Dep}), we also obtain 
\begin{align}
\label{Xep-2}
X_{\eps} &= \frac{1}{2} A - \eps \left(\frac{1}{2} \p \xi^{-1} \p \eta A - \frac{1}{8} \p \xi A \right) - \eps^3 \left( \frac 1 8 \pxi 1 \p \eta^2 A 
+\frac 1 {384}\p\xi^3 A 
-\frac 1 {48}\p\xi \p\eta A \right),\\
\nonumber
Y_{\eps} &= -\frac{1}{2} A + \eps \left(\frac{1}{2} \p \xi^{-1} \p \eta A + \frac{1}{8} \p \xi A \right) - \eps^2 \left( \frac{1}{2} \partial_{\eta} A \right) \\
\label{Yep-2}
& \qquad \qquad + \eps^3 \left( \frac 3 8 \pxi 1 \p \eta^2 A 
- \frac 1 {384}\p\xi^3 A 
+\frac 5 {48}\p\xi \p\eta A \right), \\
\nonumber
D_{\eps} &= \frac{1}{2} A + \eps \left(\frac{1}{2} \p \xi^{-1} \p \eta A + \frac{1}{8} \p \xi A \right) + \eps^2 \left( \frac{1}{2} \partial_{\eta} A \right) \\
\label{Dep-2}
& \qquad \qquad  + \eps^3 \left( \frac 3 8 \pxi 1 \p \eta^2 A - \frac 1 {384}\p\xi^3 A 
+\frac 5 {48}\p\xi \p\eta A \right),
\end{align}
Finally, substituting decompositions (\ref{Uep}), (\ref{Lep}), (\ref{Xep-2}), 
(\ref{Yep-2}), and (\ref{Dep-2}) into (\ref{Eeq}) gives the expansion:
\begin{align*}
& \eps (c_1^*)^2 \p \xi A - \frac{1}{2} \eps^2 (c_1^*)^2 \p \xi^2 A 
+ \eps^3 \left(-2 c_1^* \p \tau A + \frac{1}{12} (c_1^*)^2 \p \xi^3 A \right) 
+ \eps^4 c_1^* \p \xi \p \tau A \\
& = \eps (c_1^*)^2 \p \xi A - \frac{1}{2} \eps^2 (c_1^*)^2 \p \xi^2 A 
+ \eps^3 \left(\frac{5}{48} (c_1^*)^2 \p \xi^3 A + (c_1^*)^2 \p \xi^{-1} \p \eta^2 A + \alpha_1 A \p \xi A \right) \\
& \qquad 
+ \eps^4 \left( -\frac{1}{96} (c_1^*)^2 \p \xi^4 A - \frac{1}{2} (c_1^*)^2 \p \eta^2 A - \frac{1}{2} \alpha_1 \p \xi (A \p \xi A) \right), 
\end{align*}
which is satisfied up to the formal order of $\mathcal{O}(\eps^5)$ if $A$ is a suitable solution of the KP-II equation (\ref{KPIIdiag}).

\subsection{Step 2. Residual terms} 

Plugging the decomposition \eqref{diagdecom} into equations of motion \eqref{diagaccx} gives  the following equations for the error terms:
\begin{align}
\begin{split}
\dot L_{m,n}&=V_{m,n}-U_{m,n},\\
\dot D_{m,n}&= U_{m+1,n+1}-V_{m,n}+Res^{D}_{m,n},\\
\dot X_{m,n}&= U_{m+1,n}-V_{m,n}+Res^{X}_{m,n},\\
\dot Y_{m,n}&= U_{m,n+1}-V_{m,n}+Res^{Y}_{m,n}, \\
\dot U_{m,n}&= c_1^2 \lb L_{m,n}- D_{m-1,n-1} - X_{m-1, n} -  Y_{m,n-1}\rb\\
& \qquad +2\alpha_1 \eps^2 \lsb L_{m,n}^2-D_{m-1,n-1}^2-X_{m-1,n}^2 + Y_{m,n-1}^2\rsb \\
& \qquad +4\alpha_1 \eps^2 \lsb L_{m,n} \efun L \lb\xi,\eta\rb-D_{m-1,n-1} \efun D\lb\xi-\eps,\eta-\eps^2\rb\rsb \\
& \qquad +4\alpha_1\eps^2 \lsb Y_{m,n-1} \efun Y\lb\xi,\eta-\eps^2\rb-X_{m-1, n} \efun X \lb\xi-\eps,\eta\rb\rsb\\
& \qquad +Res^U_{m,n},\\
\dot V_{m,n} &= c_1^2 \lb D_{m,n} -L_{m,n}+ X_{m, n} + Y_{m,n}\rb\\
& \qquad +2\alpha_1 \eps^2 \lsb D_{m,n}^2-L_{m,n}^2+X_{m,n}^2-Y_{m,n}^2\rsb \\
& \qquad +4\alpha_1 \eps^2 \lsb D_{m,n} \efun D -L_{m,n} \efun L + X_{m, n} \efun X -Y_{m,n} \efun Y \rsb+Res^V_{m,n},
\end{split}
\label{dfpu}
\end{align}
where 
\begin{align*}
Res^{D}_{m,n} &:= \efun U(\xi+\eps,\eta+\eps^2) - \efun V(\xi,\eta) + \eps c_1^* \p \xi \efun D - \eps^3 \p \tau \efun D, \\
Res^{X}_{m,n} &:= \efun U(\xi+\eps,\eta) - \efun V(\xi,\eta) + \eps c_1^* \p \xi \efun X - \eps^3 \p \tau \efun X, \\
Res^{Y}_{m,n} &:= \efun U(\xi,\eta+\eps^2) - \efun V(\xi,\eta) + \eps c_1^* \p \xi \efun Y - \eps^3 \p \tau \efun Y, 
\end{align*}
and the residuals $Res^U$ and $Res^V$ are computed from the residual terms 
of equations (\ref{Eeq}) and (\ref{Feq}). Similarly to Lemma \ref{lres}, 
the residual terms are controlled in the $\ell^2(\mathbb{Z}^2)$ norm 
if $A$ is a smooth solution of the KP-II equation (\ref{KPIIdiag}). 
This estimate is summarized in the following lemma, which we give without proof. 

\begin{lemma}\label{diag-residual}
	Let $A \in C^0(\mathbb{R},H^s)$ be a solution to the KP-II equation (\ref{KPIIdiag}) with $s \geq 9$. There is a positive constant $C$ that depend on $A$ such that for all $\eps \in (0,1]$, we have
	\begin{align}
	\begin{split}
	\elt{Res^{D}}+	\elt{Res^{X}}+	\elt{Res^{Y}}+\elt{Res^{U}}+	\elt{Res^{V}}\leq C \eps^{\frac 7 2}.
	\end{split}
	\end{align}
\end{lemma}

\subsection{Step 3. Energy estimates}

In order to control the growth of the approximation error from solutions to system \eqref{dfpu}, we introduce the following energy function,
\begin{align}\label{E-diag}
\begin{split}
E(t)=&\sum_{m,n} \frac 1 2 \lb U^2_{m,n}+V^2_{m,n}\rb+\frac{1}{2} c_1^2 
\lb L^2_{m,n}+D^2_{m,n}+X^2_{m,n}+Y^2_{m,n}\rb\\
&+\sum_{m,n} 2\alpha_1 \eps^2 \lb L^2_{m,n}\efun L+D^2_{m,n}\efun D+X^2_{m,n}\efun X-Y^2_{m,n}\efun Y\rb\\
&+\sum_{m,n} \frac 2 3 \alpha_1 \eps^2 \lb L^3_{m,n}+D^3_{m,n}+X^3_{m,n}-Y^3_{m,n}\rb.
\end{split}
\end{align}
Similarly to the proof of Lemma \ref{Eest}, the energy is coercive with respect to the $\ell^2$ norm of the perturbations if $\eps$ is sufficiently small and the perturbations are not large in the $\ell^2$ norm. The only difference between the expansions (\ref{Lep}), (\ref{Xep-2}), (\ref{Yep-2}), and (\ref{Dep-2}) 
from the expansion (\ref{Ue}) is that the former involve three derivatives of $A$ 
and one derivative of $\p \xi^{-1} \p \eta A$, whereas the latter involves 
two derivatives of $A$ and one derivative of $\p \xi^{-1} \p \eta A$. 
This modifies the statement of the following lemma, which we give without proof.

\begin{lemma}\label{diag-coer}
		Let $A \in C^0\lb \lsb -\tau_0, \tau_0\rsb,H^{s+1}\lb\R^2\rb\rb$ 
	and $\p \xi^{-1} A \in C^0\lb \lsb -\tau_0, \tau_0\rsb,H^{s}\lb\R^2\rb\rb$ with $s > 3$ and assume that $E(t) \leq E_0$ for some $\eps$-independent constant $E_0 > 0$ for every $t \in [-\tau_0 \eps^{-3}, \tau_0 \eps^{-3}]$. 
	There exists some constants $\eps_0>0$ and $K_0 >0$ that depend on $A$ such that 
	\begin{equation}
	\label{energy-coercive-diag}
	\elt U^2+\elt V^2+ \elt{L}^2+ \elt{D}^2+ \elt{X}^2+ \elt{Y}^2 \leq 2 K_0 E(t),
	\end{equation}
	for each $\eps\in\lb 0,\eps_0\rb$ and $t\in [-\tau_0 \eps^{-3}, \tau_0 \eps^{-3}]$.
\end{lemma}

Finally, the growth of the energy \eqref{E-diag} is estimated from 
the balance equation:
\begin{align}
\begin{split}
E'(t)= &\sum_{m,n} U_{m,n}Res^U_{m,n}+V_{m,n}Res^V_{m,n}\\
&+c_1^2\sum_{m,n}  D_{m,n}Res^D_{m,n}+X_{m,n}Res^X_{m,n}+Y_{m,n}Res^Y_{m,n}\\
&+4\alpha_1 \eps^2 \sum_{m,n}   D_{m,n}^2Res^D_{m,n}+X_{m,n}^2Res^X_{m,n} -Y_{m,n}^2Res^Y_{m,n}\\
&+4\alpha_1 \eps^2 \sum_{m,n}  D_{m,n} \efun D Res^D_{m,n}
 + X_{m,n} \efun X Res^X_{m,n} - Y_{m,n} \efun Y Res^Y_{m,n} \\
&+2\alpha_1 \eps^2 \sum_{m,n}  L_{m,n}^2\efun {\dot L}+D_{m,n}^2\efun {\dot D}+X_{m,n}^2\efun {\dot X}-Y_{m,n}^2\efun {\dot Y}, 
\end{split}
\end{align}
where the dot denotes the derivative in $t$ of the function of 
$\xi = \eps (m - c_1^*t)$, $\eta = \eps^2 n$, and $\tau = \eps^3 t$.
In view of Lemmas \ref{diag-residual} and \ref{diag-coer}, 
similar to the proof of Lemma \ref{lemma-growth}, we can obtain a bound 
on the growth of the energy. This bound is summarized in the following 
lemma, which we give without proof. 

\begin{lemma}
	\label{lemma-growth-diag}
	Let $A \in C^0\lb \lsb -\tau_0, \tau_0\rsb,H^{s+9}\lb\R^2\rb\rb$ be a solution to the KP-II equation \eqref{KPIIdiag} with $s \geq 0$ in the class of functions of Lemma \ref{KPIIwellpos} and assume that $E(t) \leq E_0$ for some $\eps$-independent constant $E_0 > 0$ for every $t \in [-\tau_0 \eps^{-3},\tau_0 \eps^{-3}]$. There exist some constants $\eps_0 > 0$ and $K_0 > 0$ that depend on $A$ such that 
	\begin{equation}
	\label{energy-balance-diag}
	\left|E'(t)\right|\leq K_0 \lb \eps^{\frac 7 2} E(t)^{\frac 1 2}  + \eps^3 E(t) \rb,
	\end{equation}	
	for each $\eps \in (0,\eps_0)$ and $t \in [-\tau_0 \eps^{-3},\tau_0 \eps^{-3}]$.
\end{lemma}

\subsection{Step 4. Bound on the approximation error}

Since the bounds (\ref{energy-balance}) and (\ref{energy-balance-diag}) coincide, application of the Gronwall's inequality in Lemma \ref{gronwall} gives the desired result of Theorem \ref{theorem-diag} exactly like 
in Step 4 of the proof of Theorem \ref{theorem-horiz}.

\section{Conclusion}
\label{sec-conclusion}

We have proved here the validity of the KP--II approximation for dynamics of transversely modulatied small-amplitude long-scale waves in a vector FPU system on a two-dimensional square lattice. The justification was performed for horizontal and vertical propagations of the waves and, under some restrictions on parameters of the FPU system, for the diagonal propagation. While the general algorithm of the justification analysis is well understood by now, 
the technical details of the justification analysis have been developed for the first time in the vector FPU systems, to the best of our knowledge.

This research opens up new directions. First, it is interesting to see 
if the justification analysis can be generalized for the vector FPU mass--spring systems with diagonal springs and for the wave propagation under an arbitrary angle with respect to the square lattice. Second, in terms of applications 
of the FPU models to the graphene materials, it is important to consider 
other two-dimensional models such as hexagonal lattices. Finally, known properties of the KP-II equation can be applied to study other problems of the nonlinear dynamics of small-amplitude waves in the two-dimensional FPU lattices 
such as the linear and nonlinear stability of periodic and solitary waves with respect to transverse modulations.

 \end{document}